\documentclass[reqno,10pt]{amsart}

\usepackage{amsmath, amsthm}
\usepackage{pdfsync}
\usepackage{parskip}
\usepackage{hyperref}
\usepackage{mathtools}
\usepackage{mdframed}
\usepackage{tikz}
\numberwithin{equation}{section}
\newtheorem{theorem}{Theorem}[section]
\newtheorem{prop}[theorem]{Proposition}

\newtheorem{cor}[theorem]{Corollary}

\theoremstyle{remark}
\newtheorem{remark}[theorem]{Remark}
\newtheorem{rmk}[theorem]{Remarks}

\def\M{\mathsf{M}}
\def\p{\mathsf{p}}

\def\ev{\mathsf{C}}
\def\cH{\mathcal{H}}
\def\cM{\mathcal{M}}
\def\J{\mathbb{J}}
\def\Ce{\mathcal{C}_{\frak{e}}}
\def\Ca{\mathcal{C}_{\frak{a}}}

\def\R{\mathbb{R}}
\def\C{\mathbb{C}}
\def\H{\mathbb{H}^{m}}

\def\N{\mathbb{N}}
\def\Nz{\mathbb{N}_0}
\def\A{\mathfrak{A}}
\def\K{\frak{K}}

\def\L{\mathcal{L}}
\def\Lis{\mathcal{L}{\rm{is}}}
\def\F{\mathfrak{F}}
\def\bF{\boldsymbol{\mathfrak{F}}}

\def\bE{\boldsymbol{E}}
\def\bu{\boldsymbol{u}}
\def\bK{\mathbb{K}}

\def\B{\mathbb{B}}

\def\Ok{\mathsf{O}_{\kappa}}

\def\vpk{\varphi_\kappa}
\def\psk{\psi_\kappa}
\def\Q{\mathsf{Q}^{m}}
\def\Qk{\mathsf{Q}^{m}_{\kappa}}

\def\cA{\mathcal{A}}
\def\cE{\mathcal{E}}
\def\bA{\bar{\mathcal{A}}}
\def\Ak{\bar{\mathcal{A}}_\kappa}
\def\bfa{\boldsymbol{\mathfrak{a}}}

\def\cB{\mathcal{B}}
\def\Bk{\mathcal{B}_\kappa}
\def\cC{\mathcal{C}}
\def\Ck{\mathcal{C}_\kappa}

\def\kb{\varphi^{\ast}_{\kappa}}
\def\kf{\psi^{\ast}_{\kappa}}

\def\pk{\pi_\kappa}

\def\Re{\mathcal{R}}
\def\Rek{\mathcal{R}_{\kappa}}
\def\Rc{\mathcal{R}^c}
\def\Rck{\mathcal{R}^c_{\kappa}}

\def\id{{\rm{id}}}
\def\supp{{\rm{supp}}}
\def\uf{{\rm{unif}}}
\def\gd{{\rm{grad}}}
\def\div{{\rm div}}

\def\Hom{{\text{Hom}}}

\def\X{\mathbb{X}}
\def\Xk{\mathbb{X}_{\kappa}}
\def\ez{\mathbb{E}_{0}}
\def\ef{\mathbb{E}_{1}}

\begin{document}

\title[Continuous maximal regularity on singular manifolds]{Continuous maximal regularity on singular manifolds and its applications}

\author[Y. Shao]{Yuanzhen Shao}
\address{Department of Mathematics,
         Vanderbilt University, 
         1326 Stevenson Center, 
         Nashville, TN 37240, USA}
\email{yuanzhen.shao@vanderbilt.edu}


\begin{abstract}
In this article, we set up the continuous maximal regularity theory for a class of linear differential operators on manifolds with singularities. These operators exhibit degenerate or singular behaviors while approaching the singular ends. Particular examples of such operators include differential operators defined on domains, which degenerate fast enough toward the boundary. Applications of the theory established herein are shown to the Yamabe flow, the porous medium equation, the parabolic $p$-Laplacian equation and the thin film equation. Some comments about the boundary blow-up problem, and waiting time phenomenon for singular or degenerate parabolic equations can also be found in this paper.
\end{abstract}
\maketitle

\section{\bf Introduction}

The main objective of this article is to establish the continuous maximal regularity for a family of degenerate or singular elliptic operators on a class of manifolds with singularities, called singular manifolds. These results generalize the work in the previous paper \cite{ShaoSim13}. The notation of singular manifolds used in this paper was first introduced by H.~Amann in \cite{Ama13}. Roughly speaking, a manifold $(\M,g)$ is singular iff it is conformal to one whose local patches are of comparable sizes, and all transition maps and curvatures have uniformly bounded derivatives, i.e., $(\M,g/\rho^2)$ has the aforementioned properties for some $\rho\in C^\infty(\M,(0,\infty))$. In \cite{DisShaoSim}, it is shown that the class of all such $(\M,g/\rho^2)$, called uniformly regular Riemannian manifolds, coincides with the family of complete manifolds with bounded geometry if we restrict ourselves to manifolds without boundary.

In \cite{Ama13b}, the author built up the $L_p$-maximal regularity for a family of second-order elliptic operators satisfying a certain ellipticity condition, called {\em uniformly strongly $\rho$-elliptic}. By this, the author means that the principal part $-\div(\ev(\vec{a},\gd u))$ of a differential operator fulfils
$$ (\ev(\vec{a},\xi) | \xi)_{g^*} \sim \rho^2 |\xi|_{g^*}^2, $$
for any cotangent field $\xi$. 
Here $\vec{a}$ is a symmetric $(1,1)$-tensor field on $(\M,g)$, and $\ev(\cdot,\cdot)$ denotes complete contraction. See Section~3 for the precise definition.
If two real-valued functions $f$ and $g$ are equivalent in the sense that $f/c\leq g\leq cf$ for some $c\geq 1$, then we write $f\sim g$.
These operators, as we can immediately observe from the above relationship, can exhibit degenerate or singular behaviors while approaching the singular ends. In \cite{Ama13b}, H.~Amann also looked at manifolds with boundary. 
We generalize this concept of {\em uniformly strong $\rho$-ellipticity} to elliptic operators of arbitrary even order acting on tensor bundles. 
A linear operator 
$$\cA:=\sum\limits_{r=0}^{2l} \ev(a_r,\nabla^r \cdot)$$ 
of order $2l$, where $a_r$ is a $(\sigma+\tau+r, \tau+\sigma)$-tensor field, is said to be {\em uniformly strongly $\rho$-elliptic}  if its principal part $\ev(a_{2l}, \nabla^{2l}\cdot)$ satisfies that for each cotangent field $\xi$ and every $(\sigma,\tau)$-tensor field $\eta$, it holds
\begin{align}
\label{S1: ellipticity}
(\ev(a_{2l},\eta\otimes(-i\xi)^{\otimes {2l}})|\eta)_g \sim \rho^{2l} |\eta|_g^2 |\xi|_{g^*}^{2l}. 
\end{align}
Moreover, in Section~3.1, we show that this ellipticity condition can actually be replaced by a weaker one, called {\em normal $\rho$-ellipticity}. But for the sake of simplicity, we still stay with {\em uniformly strong $\rho$-ellipticity} stated above in the introduction. 

By imposing some mild regularity condition on the coefficients $a_r$ of $\cA$, called {\em $s$-regularity}, we are able to prove the following result.
\begin{theorem}
\label{S1: main theorem}
Let $s\in \R_+\setminus\N$ and $\vartheta\in \R$. Suppose that a $2l$-th order linear differential operator $\cA$ is {\em uniformly strongly $\rho$-elliptic} and {\em $s$-regular}.  Then
$$
\cA\in\cH(bc^{s+2l,\vartheta}(\M, V^\sigma_\tau),bc^{s,\vartheta}(\M, V^\sigma_\tau)).
$$
\end{theorem}
Here $u\in bc^{s,\vartheta}(\M, V^\sigma_\tau)$ iff $\rho^\vartheta u$ is a $(\sigma,\tau)$-tensor field with {\em little H\"older} continuity of order $s$. The precise definition of weighted {\em little H\"older} spaces will be presented in Section~2.2. 
An operator ${A}$ is said to belong to the class $\cH(E_1,E_0)$ for some densely embedded Banach couple $E_1\overset{d}{\hookrightarrow}E_0$, if $-A$ generates a strongly continuous analytic semigroup on $E_0$ with $dom(-A)=E_1$.
By means of a well-known result of G.~Da~Prato, P.~Grisvard \cite{DaPra79} and S.~Angenent \cite{Ange90}, this theorem yields the continuous maximal regularity property of $\cA$.
Theorem~\ref{S1: main theorem} generalizes the results in \cite{ShaoSim13} in the sense that taking $\rho\sim {\bf 1}_\M$, it agrees with the continuous maximal regularity theory in \cite{ShaoSim13} on uniformly regular Riemannian manifolds. 
Note that the theory established therein can somehow be considered as a generalization of the work on manifolds with cylindrical ends by R.B. Melrose \cite{Mel81, Mel93} and his collaborators. 

The proof of the main theorem follows from the ideas in \cite[Section~3]{ShaoSim13}. The cornerstone of this proof is a  properly defined retraction and coretraction system between weighted function spaces over manifolds and in Euclidean spaces, see Section~2.3 for the precise definition.  This system enables us to apply the well-studied elliptic and parabolic theory in Euclidean spaces and translate it into the manifold framework. 

One important feature of this paper is that it is application-oriented. We apply Theorem~\ref{S1: main theorem} to several well-known evolution equations. These results are stated in Section~4. As an example in geometric analysis, we show the well-posedness of the Yamabe flow on singular manifolds, in particular, with unbounded curvature. The Yamabe flow arises as an alternative approach to the famous Yamabe problem. 
It was introduced by R.~Hamilton shortly after the Ricci flow, and studied extensively by many authors afterwards. The reader may consult \cite[Section~5]{ShaoSim13} for a brief historical account of this problem.

In addition to its application to geometric analysis, we also apply the main theorem to two well-known relatives of the heat equation, namely, the porous medium equation and the parabolic $p$-Laplacian equation, on a singular manifold $(\M,g)$.  J.L.~V\'azquez \cite{Vaz92, Vaz07} proved  existence and uniqueness of non-negative weak solutions of Dirichlet problems for the porous medium equation. In a landmark article \cite{DasHam98}, P.~Daskalopoulos and R.~Hamilton showed existence and uniqueness of smooth solutions for the porous medium equation, and the smoothness of the free boundary, namely, the boundary of the support of the solution, under mild assumptions on the initial data. In the past decade, there has been rising interest in investigating the porous medium equation on Riemannian manifolds. See \cite{ChaLee12, Dek08, KisShtZla08, OttoWest05, Xu12, Zhu13} for example. To the best of the author's knowledge, research in this direction is all restricted to the case of complete, or even compact, manifolds. The result that we state in Section~4.1 seems to be the first one concerning existence and uniqueness of  solutions to the porous medium equation on manifolds with singularities. 
Following the same strategy, we study the $p$-Laplacian equation, a nonlinear counterpart of the Laplacian equation, which is probably one of the best known examples of degenerate or singular equations in divergence form. In Section~4.3, we explore the parabolic $p$-Laplacian equation on a singular manifold $(\M,g)$:
\begin{equation*}
\left\{\begin{aligned}
\partial_t u -\div_g (|\gd_g u|_g^{p-2}\gd_g u)&=f;  \\
u(0)&=u_0 . &&
\end{aligned}\right.
\end{equation*}
Here $p>1$ with $p\neq 2$. This problem has been studied extensively on Euclidean spaces. The two books \cite{Dib93, DibGiaVesp12} contain a detailed analysis and a  historical account of this problem. There are several generalizations of the elliptic $p$-Laplacian equation on Riemannian manifolds. But fewer have been achieved for its parabolic version above. See \cite{Dek08} for instance. The study of these nonlinear heat equations also produces intriguing applications for degenerate boundary value problems or boundary blow-up problems. In Section~3.2, it is shown that any smooth domain $(\Omega , g_m)$ in $\R^m$ with compact boundary can be realized as a singular manifold, where $g_m$ denotes the Euclidean metric. Then we can prove the local existence and uniqueness of solutions to the following boundary blow-up problem for $1<p<2$ in {\em little H\"older} spaces. 
\begin{equation}
\label{S1: boundary blowup}
\left\{\begin{aligned}
\partial_t u -\div (|D u|^{p-2}D u)&=0 &&\text{on}&& \Omega_T;  \\
u&= \infty &&\text{on}&&\partial\Omega_T;\\
u(0)&=u_0&&\text{on}&&\Omega   &&
\end{aligned}\right.
\end{equation}
as long as the initial data $u_0$ belongs to a properly chosen open subset in some {\em H\"older} space. Here $\Omega_T:=\Omega\times (0,T)$. See Remark~\ref{S4.3: boundary-blowup} for more details.

Another application of the continuous maximal regularity theory established in this paper concerns  parabolic equations with higher order degeneracy on domains with compact boundary. The order of the degeneracy is measured by the rate of decay in the ellipticity condition while approaching the boundary. See Theorem~\ref{S3: MR-domain} for a precise description. This result extends the work in \cite{ForMetPall11, Vesp89} to unbounded domains and to higher order elliptic operators.  In the last subsection, we prove a local existence and uniqueness theorem for a generalized multidimensional thin film equation 
\begin{equation}
\label{S1: gener-TFE}
\left\{\begin{aligned}
\partial_t u +\div (u^n D \Delta u +\alpha_1 u^{n-1}\Delta u D u +\alpha_2 u^{n-2} |D u|^2 D u)&=f  &&\text{on}&&\Omega_T; \\
u(0)&=u_0  &&\text{on }&&\Omega &&
\end{aligned}\right.
\end{equation}
if the initial data decays sufficiently fast to the boundary of its support. Here $\alpha_1,\alpha_2$ are two constants, $n>0$, and $\Omega\subset\R^m$ is a sufficiently smooth domain. This generalized model was first investigated by J.R. King in \cite{King01} in the one dimensional case. Later, a multidimensional counterpart has been studied with periodic boundary condition on cubes in \cite{BouHilRak08}. 
An interesting waiting-time phenomenon can be observed from our approach. The mathematical investigation of the thin film equation was initiated by F. Bernis and A. Friedman in \cite{BerFrd90}. An intriguing feature of free boundary problems associated with degenerate parabolic equations is the waiting-time phenomenon of the supports of the solutions. This phenomenon has been widely observed and studied by many mathematicians. See \cite{DalGiaGrun01, DiaNagShm96, GiaGrun06, Grun02, Shish07} for example. The waiting-time phenomenon for the case $\alpha_1,\alpha_2=0$, the original thin film equation, has been explored in several of the papers listed above. Our result extends the results in the above literature for the generalized system~\eqref{S1: gener-TFE}.

It is worthwhile mentioning that sometimes to establish the theory for nonlinear parabolic equations, in some sense, is easier than that for linear equations. This surprising phenomenon can be observed from the heat equation $\partial_t u -\Delta_g u=0$. Note that $\Delta_g=\ev(g^*,\nabla^2\cdot) $. In this case, the principal symbol of $\Delta_g$ can be computed as
$$\ev(g^*, \xi^{\otimes 2}) = |\xi|_{g^*}^2. $$
The power of the weight function $\rho$ is different from \eqref{S1: ellipticity} in this case. This breaks the uniform ellipticity conditions of the local expressions of the corresponding differential operators as we can observe from the discussion in Section~3 below. Linear differential operators with degeneracy other than $\rho^2$ have been investigated by many authors, including B.-W. Schulze \cite{Sch94, Sch97} and his collaborators. But these results depend heavily on the specific geometric structure near the singular ends. In two subsequent papers \cite{ShaoPre, ShaoPre2}, we treat second order differential operators with a different order of degeneracy from that in \eqref{S1: ellipticity}. In the nonlinear case, the nonlinearities  sometimes give rise to the right power of $\rho$ in \eqref{S1: ellipticity}, as is shown by the examples in Section~4.

This paper is organized as follows. In the rest of this introductory section, we give the precise definitions of uniformly regular Riemannian manifolds and singular manifolds. Section~2 is the stepstone to the theory of differential operators, where we define the weighted {\em H\"older} and {\em little H\"older} spaces on singular manifolds and introduce some of their properties, following the work of H.~Amann in \cite{Ama13, AmaAr}. In Section~3, we establish continuous maximal regularity for differential operators satisfying the conditions in Theorem~\ref{S1: main theorem}. In the last section, several applications of continuous maximal regularity theory are presented.

{\textbf {Assumptions on manifolds:}}
Following H.~Amann \cite{Ama13, AmaAr},
let $(\M,g)$ be a $C^{\infty}$-Riemannian manifold of dimension $m$ with or without boundary endowed with $g$ as its Riemannian metric such that its underlying topological space is separable. An atlas $\A:=(\Ok,\vpk)_{\kappa\in \K}$ for $\M$ is said to be normalized if 
\begin{align*}
\vpk(\Ok)=
\begin{cases}
\Q, \hspace*{1em}& \Ok\subset\mathring{\M},\\
\Q\cap\H, &\Ok\cap\partial\M \neq\emptyset,
\end{cases}
\end{align*}
where $\H$ is the closed half space $\bar{\R}^+ \times\R^{m-1}$ and $\Q$ is the unit cube at the origin in $\R^m$. We put $\Qk:=\vpk(\Ok)$ and  $\psk:=\vpk^{-1}$. 
\smallskip\\
The atlas $\A$ is said to have \emph{finite multiplicity} if there exists $K\in \N $ such that any intersection of more than $K$ coordinate patches is empty. Put
\begin{align*}
\mathfrak{N}(\kappa):=\{\tilde{\kappa}\in\K:\mathsf{O}_{\tilde{\kappa}}\cap\Ok\neq\emptyset \}.
\end{align*} 
The finite multiplicity of $\A$ and the separability of $\M$ imply that $\A$ is countable.
\smallskip\\
An atlas $\A$ is said to fulfil the \emph{uniformly shrinkable} condition, if it is normalized and there exists $r\in (0,1)$ such that $\{\psk(r{\Qk}):\kappa\in\K\}$ is a cover for ${\M}$. 
\smallskip\\
Following H.~Amann \cite{Ama13, AmaAr}, we say that $(\M,g)$ is a {\bf{uniformly regular Riemannian manifold}} if it admits an atlas $\A$ such that
\begin{itemize}
\item[(R1)] $\A$ is uniformly shrinkable and has finite multiplicity. If $\M$ is oriented, then $\A$ is orientation preserving.
\item[(R2)] $\|\varphi_{\eta}\circ\psk \|_{k,\infty}\leq c(k) $, $\kappa\in\K$, $\eta\in\mathfrak{N}(\kappa)$, and $k\in{\N}_0$.
\item[(R3)] $\kf g\sim g_m $, $\kappa\in\K$. Here $g_m$ denotes the Euclidean metric on ${\R}^m$ and $\kf g$ denotes the pull-back metric of $g$ by $\psk$.
\item[(R4)] $\|\kf g\|_{k,\infty}\leq c(k)$, $\kappa\in\K$ and $k\in\Nz$.
\end{itemize}
Here $\|u\|_{k,\infty}:=\max_{|\alpha|\leq k}\|\partial^{\alpha}u\|_{\infty}$, and it is understood that a constant $c(k)$, like in (R2), depends only on $k$. An atlas $\A$ satisfying (R1) and (R2) is called a \emph{uniformly regular atlas}. (R3) reads as
\begin{center}
$|\xi|^2/c\leq \kf g(x)(\xi,\xi) \leq{c|\xi|^2}$,\hspace{.5em} for any $x\in \Qk,\xi\in \R^m, \kappa\in\K$ and some $c\geq{1}$.
\end{center}
In \cite{DisShaoSim}, we have shown that the class of uniformly regular Riemannian manifolds coincides with the family of complete Riemannian manifolds with bounded geometry, when $\partial\M=\emptyset$.

Assume that $\rho\in C^{\infty}(\M,(0,\infty))$. Then $(\rho,\K)$ is a {\em singularity datum} for $\M$ if
\begin{itemize}
\item[(S1)] $(\M,g/\rho^2)$ is a uniformly regular Riemannian manifold.
\item[(S2)] $\A$ is a uniformly regular atlas.
\item[(S3)] $\|\kf\rho\|_{k,\infty}\leq c(k)\rho_{\kappa}$, $\kappa\in\K$ and $k\in\N_0$, where $\rho_{\kappa}:=\rho(\psk(0))$.
\item[(S4)] $\rho_{\kappa}/c\leq \rho(\p)\leq c\rho_{\kappa}$, $\p\in\Ok$ and $\kappa\in\K$ for some $c\geq 1$ independent of $\kappa$.
\end{itemize}
Two {\em singularity data} $(\rho,\K)$ and $(\tilde{\rho},\tilde{\K})$ are equivalent, if
\begin{itemize}
\item[(E1)] $\rho\sim \tilde{\rho}$.
\item[(E2)] card$\{\tilde{\kappa}\in\tilde{\K}:\mathsf{O}_{\tilde{\kappa}}\cap\Ok\neq\emptyset\}\leq c$, $\kappa\in\K$.
\item[(E3)] $\|\varphi_{\tilde{\kappa}}\circ\psk\|_{k,\infty}\leq{c(k)}$, $\kappa\in\K$, $\tilde{\kappa}\in\tilde{\K}$ and $k\in{\N}_0$
\end{itemize}
We write the equivalence relationship as $(\rho,\K)\sim(\tilde{\rho},\tilde{\K})$. (S1) and (E1) imply that 
\begin{align}
\label{section 1:singular data}
1/c\leq \rho_{\kappa}/\tilde{\rho}_{\tilde{\kappa}}\leq c,\hspace*{.5em} \kappa\in\K,\hspace*{.5em} \tilde{\kappa}\in\tilde{\K}\text{ and }\mathsf{O}_{\tilde{\kappa}}\cap\Ok\neq \emptyset.
\end{align}
{\em A singularity structure}, $\mathfrak{S}(\M)$, for $\M$ is a maximal family of equivalent {\em singularity data}. A {\em singularity function} for $\mathfrak{S}(\M)$ is a function $\rho\in C^{\infty}(\M,(0,\infty))$ such that there exists an atlas $\A$ with $(\rho,\A)\in\mathfrak{S}(\M)$. The set of all {\em singularity functions} for $\mathfrak{S}(\M)$ is the {\em singular type}, $\mathfrak{T}(\M)$, for $\mathfrak{S}(\M)$. By a {\bf{singular manifold}} we mean a Riemannian manifold $\M$ endowed with a singularity structure $\mathfrak{S}(\M)$. Then $\M$ is said to be \emph{singular of type} $\mathfrak{T}(\M)$. If $\rho\in\mathfrak{T}(\M)$, then it is convenient to set $[\![\rho]\!]:=\mathfrak{T}(\M)$ and to say that $(\M,g;\rho)$ is a singular manifold. A singular manifold is a uniformly regular Riemannian manifold iff $\rho\sim {\bf 1}_{\M}$. 
\smallskip\\
We refer to \cite{Ama13b, Ama14} for examples of uniformly regular Riemannian manifolds and singular manifolds.
\smallskip\\
A singular manifold $\M$ with a uniformly regular atlas $\A$ admits a \emph{localization system subordinate to} $\A$, by which we mean a family $(\pi_{\kappa},\zeta_{\kappa})_{\kappa\in\K}$ satisfying:
\begin{itemize}
\item[(L1)] $\pk \in\mathcal{D}(\Ok,[0,1])$ and $(\pi_{\kappa}^{2})_{\kappa\in \K} $ is a partition of unity subordinate to $\A$.
\item[(L2)] $\zeta_\kappa :={\kb}\zeta$ with $\zeta\in\mathcal{D}(\Q,[0,1])$ satisfying $\zeta|_{\supp({{\kf}\pi_{\kappa}})}\equiv 1$, $\kappa\in\K$.
\item[(L3)] $\|\kf  \pk \|_{k,\infty} \leq c(k) $, for $\kappa\in\K$, $k\in \Nz$.
\end{itemize}
The reader may refer to \cite[Lemma~3.2]{Ama13} for a proof. 

Lastly, for each $k\in\N$, the concept of {\bf{$C^k$-uniformly regular Riemannian manifold}} is defined by modifying (R2), (R4) and (L1)-(L3) in an obvious way. Similarly, {\bf{$C^k$-singular manifolds}} are defined by replacing the smoothness of $\rho$ by $\rho\in C^k(\M,(0,\infty))$ and altering (S1)-(S3) accordingly.

\textbf{Notations:} Let $\bK\in \{\R,\C \}$. $\N_0$ is the set of all natural numbers including $0$.
For any interval $I$ containing $0$, $\dot{I}:= I \setminus\{0\}$. 
\smallskip\\
For any two Banach spaces $X,Y$, $X\doteq Y$ means that they are equal in the sense of equivalent norms. The notation $\Lis(X,Y)$ stands for the set of all bounded linear isomorphisms from $X$ to $Y$.
\smallskip\\
For any Banach space $E$, we abbreviate $\F(\R^m,E)$ to $\F(E)$, where $\F$ stands for any function space defined in this article. 
The precise definitions for these function spaces will be presented in Section 2. 
\smallskip\\
Let $\| \cdot \|_\infty$, $\| \cdot \|_{s,\infty}$, $\|\cdot \|_p$ and $\|\cdot \|_{s,p}$ denote the usual norm of the Banach spaces $BC(E)$($L_\infty(E)$), $BC^s(E)$, $L_p(E)$ and $W^s_p(E)$, respectively. 
\smallskip\\
We denote $\bK$-valued function spaces with domain $U\in \{\M,\Omega\}$ by $\F(U)$ if $\Omega\subset\R^m$ with $\Omega\neq \R^m$.

\section{\bf Preliminaries}

In this section, we define the weighted function spaces on singular manifolds, following the work of H.~Amann in \cite{Ama13, AmaAr}.
\smallskip\\
Let $\mathbb{A}$ be a countable index set. Suppose $E_{\alpha}$ is for each $\alpha\in\mathbb{A}$ a locally convex space. We endow $\prod_{\alpha}E_{\alpha}$ with the product topology, that is, the coarsest topology for which all projections $pr_{\beta}:\prod_{\alpha}E_{\alpha}\rightarrow{E_{\beta}},(e_{\alpha})_{\alpha}\mapsto{e_{\beta}}$ are continuous. By $\bigoplus_{\alpha}E_{\alpha}$ we mean the vector subspace of $\prod_{\alpha}E_{\alpha}$ consisting of all finitely supported elements, equipped with the inductive limit topology, that is, the finest locally convex topology for which all injections $E_{\beta}\rightarrow\bigoplus_{\alpha}E_{\alpha}$ are continuous. 

\subsection{\bf Tensor bundles}
Suppose $(\M,g;\rho)$ is a singular manifold.
Given $\sigma,\tau\in\N_0$, 
$$T^{\sigma}_{\tau}{\M}:=T{\M}^{\otimes{\sigma}}\otimes{T^{\ast}{\M}^{\otimes{\tau}}}$$ 
is the $(\sigma,\tau)$-tensor bundle of $\M$, where $T{\M}$ and $T^{\ast}{\M}$ are the tangent and the cotangent bundle of ${\M}$, respectively.
We write $\mathcal{T}^{\sigma}_{\tau}{\M}$ for the $C^{\infty}({\M})$-module of all smooth sections of $T^{\sigma}_{\tau}\M$,
and $\Gamma(\M,T^{\sigma}_{\tau}{\M})$ for the set of all sections.

For abbreviation, we set $\J^{\sigma}:=\{1,2,\ldots,m\}^{\sigma}$, and $\J^{\tau}$ is defined alike. Given local coordinates $\varphi=\{x^1,\ldots,x^m\}$, $(i):=(i_1,\ldots,i_{\sigma})\in\J^{\sigma}$ and $(j):=(j_1,\ldots,j_{\tau})\in\J^{\tau}$, we set
\begin{align*}
\frac{\partial}{\partial{x}^{(i)}}:=\frac{\partial}{\partial{x^{i_1}}}\otimes\cdots\otimes\frac{\partial}{\partial{x^{i_{\sigma}}}}, \hspace*{.5em} \partial_{(i)}:=\partial_{i_{1}}\circ\cdots\circ\partial_{i_{\sigma}} \hspace*{.5em} dx^{(j)}:=dx^{j_1}\otimes{\cdots}\otimes{dx}^{j_{\tau}}
\end{align*}
with $\partial_{i}=\frac{\partial}{\partial{x^i}}$. The local representation of 
$a\in \Gamma(\M,T^{\sigma}_{\tau}{\M})$ with respect to these coordinates is given by 
\begin{align}
\label{local}
a=a^{(i)}_{(j)} \frac{\partial}{\partial{x}^{(i)}} \otimes dx^{(j)} 
\end{align}
with coefficients $a^{(i)}_{(j)}$ defined on $\Ok$.
\smallskip\\
We denote by $\nabla=\nabla_g$ the Levi-Civita connection on $T{\M}$. It has a unique extension over $\mathcal{T}^{\sigma}_{\tau}{\M}$ satisfying, for $X\in\mathcal{T}^1_0{\M}$,
\begin{itemize}
\item[(i)] $\nabla_{X}f=\langle{df,X}\rangle$, \hspace{1em}$f\in{C^{\infty}({\M})}$,
\item[(ii)] $\nabla_{X}(a\otimes{b})=\nabla_{X}a\otimes{b}+a\otimes{\nabla_{X}b}$, \hspace{1em}$a\in\mathcal{T}^{{\sigma}_1}_{{\tau}_1}{\M}$, $b\in\mathcal{T}^{{\sigma}_2}_{{\tau}_2}{\M}$,
\item[(iii)] $\nabla_{X}\langle{a,b}\rangle=\langle{\nabla_{X}a,b}\rangle+\langle{a,\nabla_{X}b}\rangle$, \hspace{1em}$a\in\mathcal{T}^{{\sigma}}_{{\tau}}{\M}$, $b\in\mathcal{T}^{{\tau}}_{{\sigma}}{\M}$,
\end{itemize}
where $\langle{\cdot,\cdot}\rangle:\mathcal{T}^{\sigma}_{\tau}{\M}\times{\mathcal{T}^{\tau}_{\sigma}{\M}}\rightarrow{C^{\infty}({\M})}$ is the extension of the fiber-wise defined duality pairing on ${\M}$, cf. \cite[Section 3]{Ama13}. Then the covariant (Levi-Civita) derivative is the linear map
\begin{center}
$\nabla: \mathcal{T}^{\sigma}_{\tau}{\M}\rightarrow{\mathcal{T}^{\sigma}_{\tau+1}{\M}}$, $a\mapsto{\nabla{a}}$
\end{center}
defined by
\begin{center}
$\langle{\nabla{a},b\otimes{X}}\rangle:=\langle{\nabla_{X}a,b}\rangle$, \hspace{1em}$b\in\mathcal{T}^{\tau}_{\sigma}{\M}$, $X\in\mathcal{T}^{1}_{0}{\M}$.
\end{center}
For $k\in{\N}_0$, we define
\begin{center}
$\nabla^k: \mathcal{T}^{\sigma}_{\tau}{\M}\rightarrow{\mathcal{T}^{\sigma}_{\tau+k}{\M}}$, $a\mapsto{\nabla^k{a}}$
\end{center}
by letting $\nabla^0 a:=a$ and $\nabla^{k+1} a:=\nabla\circ\nabla^k a$.
We can also extend the Riemannian metric $(\cdot|\cdot)_g$ from the tangent bundle to any $(\sigma,\tau)$-tensor bundle $T^{\sigma}_{\tau}{\M}$ such that $(\cdot|\cdot)_g:=(\cdot|\cdot)_{g^\tau_\sigma}:T^{\sigma}_{\tau}{\M}\times{T^{\sigma}_{\tau}{\M}}\rightarrow \bK $ by 
$$(a|b)_g =g_{(i)(\tilde{i})} g^{(j)(\tilde{j})}	a^{(i)}_{(j)}  \bar{b}^{(\tilde{i})}_{(\tilde{j})}$$
in every coordinate with $(i),(\tilde{i})\in \J^\sigma$, $(j),(\tilde{j})\in \J^\tau$ and 
$$ g_{(i)(\tilde{i})}:= g_{i_1}g_{\tilde{i}_1}\cdots g_{i_\sigma}g_{\tilde{i}_\sigma},\quad g^{(j)(\tilde{j})}:= g^{j_1} g^{\tilde{j}_1}\cdots g^{j_\tau} g^{\tilde{j}_\tau}.$$ 
In addition,
\begin{center}
$|\cdot|_g:=|\cdot|_{g^\tau_\sigma}:\mathcal{T}^{\sigma}_{\tau}{\M}\rightarrow{C^{\infty}}({\M})$, $a\mapsto\sqrt{(a|a)_g}$
\end{center}
is called the (vector bundle) \emph{norm} induced by $g$.
\smallskip\\
We assume that $V$ is a $\bK$-valued tensor bundle on $\M$ and $E$ is a $\bK$-valued vector space, i.e.,
\begin{center}
$V=V^{\sigma}_{\tau}:=\{T^{\sigma}_{\tau}\M, (\cdot|\cdot)_g\}$,\hspace{.5em} and\hspace{.5em} $E=E^{\sigma}_{\tau}:=\{\bK^{m^{\sigma}\times m^{\tau}},(\cdot|\cdot)\}$, 
\end{center}
for some $\sigma,\tau\in\N_0$. Here $(a|b):=$trace$(b^{\ast}a)$ with $b^{\ast}$ being the conjugate matrix of $b$. By setting $N=m^{\sigma+\tau}$ , we can identify $\F^s(\M,E)$ with $\F^s(\M)^N$. 

Recall that for any $a\in V^\sigma_{\tau+1}$, then $a^\sharp \in V^{\sigma+1}_\tau$ is defined by
$$(a^\sharp)^{(i;k)}_{(j)}:=g^{kl} a^{(i)}_{(j;l)}, \quad (i)\in\J^\sigma,\quad(j)\in\J^\tau,\quad  k,l\in\J^1. $$
We have 
$|a^\sharp|_{g^\tau_{\sigma+1}}=|a|_{g^{\tau+1}_\sigma}.$ 
Given any $a\in V^{\sigma+1}_\tau$, $a_\flat \in V^\sigma_{\tau+1}$ is defined as
$$(a_\flat)^{(i)}_{(j;k)}:=g_{kl} a^{(i;l)}_{(j)}.$$
Similarly, we have
$|a_\flat|_{g^{\tau+1}_\sigma}=|a|_{g_{\sigma+1}^\tau}. $
\smallskip\\
Suppose that $\sigma+\tau\geq 1$. We put for $a\in V^\sigma_\tau$ and $\alpha_i\in T^*\M$, $\beta^j\in T\M$
$$(G^\tau_\sigma a)(\alpha_1,\cdots,\alpha_\tau; \beta^1,\cdots,\beta^\sigma):=a((\beta^1)_\flat ,\cdots,(\beta^\sigma)_\flat;(\alpha_1)^\sharp,\cdots,(\alpha_\tau)^\sharp ).$$
Then it induces a conjugate linear bijection
$$G^\tau_\sigma: V^\sigma_\tau\rightarrow V^\tau_\sigma, \quad (G^\tau_\sigma)^{-1}=G^\sigma_\tau. $$
Consequently, for $a,b\in V$
$$(a|b)_g= \langle a, G^\tau_\sigma b\rangle. $$
From this, it is easy to show that
\begin{align}
\label{S2: G^tau_sisgma}
|G^\tau_\sigma a|_{g^\sigma_\tau}=|a|_{g^\tau_\sigma}.
\end{align}
Throughout the rest of this paper, unless stated otherwise, we always assume that 
\begin{mdframed}
\begin{itemize}
\item $(\M,g;\rho)$ is a singular manifold.
\item $\rho\in \mathfrak{T}(\M)$, $s\geq 0$, and $\vartheta\in\R$.
\item $(\pk,\zeta_{\kappa})_{\kappa\in\K}$ is a localization system subordinate to $\mathfrak{A}$.
\item $\sigma,\tau\in \Nz$, $V=V^{\sigma}_{\tau}:=\{T^{\sigma}_{\tau}\M, (\cdot|\cdot)_g\}$, $E=E^{\sigma}_{\tau}:=\{\bK^{m^{\sigma}\times m^{\tau}},(\cdot|\cdot)\}$.
\end{itemize}
\end{mdframed}
In \cite[Lemma~3.1]{Ama13}, it is shown that $\M$ satisfies the following properties:
\begin{itemize}
\item[(P1)] ${\kf}g\sim \rho^2_{\kappa}g_m$ and ${\kf}g^{\ast}\sim \rho^{-2}_{\kappa}g_m$, where $g^{\ast}$ is the induced contravariant metric.
\item[(P2)] $\rho^{-2}_{\kappa}\|{\kf}g\|_{k,\infty}+\rho^2_{\kappa}\|{\kf}g^{\ast}\|_{k,\infty}\leq c(k)$, $k\in\N_0$ and $\kappa\in\K$.
\item[(P3)] For $\sigma,\tau\in\N_0$ given, then 
\begin{center}
${\kf}(|a|_g)\sim \rho^{\sigma-\tau}_{\kappa}|{\kf}a|_{g_m}$, \hspace{1em}$a\in \mathcal{T}^{\sigma}_{\tau}\M$,
\end{center}
and
\begin{center}
$|{\kb}b|_g \sim \rho^{\sigma-\tau}_\kappa \kb(|b|_{g_m})$, \hspace{1em}$b\in \mathcal{T}^{\sigma}_{\tau}\Qk$.
\end{center}
\end{itemize}
\smallskip 
For $K\subset \M$, we put $\K_{K}:=\{\kappa\in \K: \Ok\cap K\neq\emptyset\}$. Then, given $\kappa\in\K$,
\begin{align*}
\Xk:=
\begin{cases}
\R^m \hspace*{1em}\text{if }\kappa\in \K\setminus \K_{\partial\M},\\
\H \hspace*{1em}\text{otherwise,}
\end{cases}
\end{align*}
endowed with the Euclidean metric $g_m$.

Given $a\in \Gamma(\M,V)$ with local representation $\eqref{local}$
we define ${\kf}a\in E$ by means of
$ {\kf}a=[a^{(i)}_{(j)}]$,
where $[a^{(i)}_{(j)}]$ stands for the $(m^{\sigma}\times m^{\tau})$-matrix with entries $a^{(i)}_{(j)}$ in the $((i),(j))$ position, with $(i)$, $(j)$ arranged lexicographically.

\subsection{\bf Weighted function spaces}
For the sake of brevity, we set $\boldsymbol{L}_{1,loc}(\X,E):=\prod_{\kappa}{L}_{1,loc}({\Xk},E)$. Then we introduce two linear maps for $\kappa\in\K$:
\begin{center}
$\Rck:{L}_{1,loc}({\M},V)\rightarrow{L}_{1,loc}({\Xk},E)$, $u\mapsto{\psk^{\ast}({\pk}u)}$,
\end{center}
and
\begin{center}
$\Rek:{L}_{1,loc}({\Xk},E)\rightarrow{L}_{1,loc}({\M},V)$, $v_{\kappa}\mapsto{\pk}{\kb}v_{\kappa}$.
\end{center}
Here and in the following it is understood that a partially defined and compactly supported tensor field is automatically extended over the whole base manifold by identifying it to be zero  outside its original domain.
We define
$$\Rc:{L}_{1,loc}({\M},V)\rightarrow\boldsymbol{L}_{1,loc}({{\R}^{m}}),\quad u\mapsto{(\Rck u)_{\kappa}},$$
and
$$\Re:\boldsymbol{L}_{1,loc}({{\R}^{m}})\rightarrow{L}_{1,loc}({\M},V),\quad (v_{\kappa})_{\kappa}\mapsto{\sum_{\kappa}\Rek v_\kappa}.$$ 
\smallskip\\
In the rest of this subsection we assume that $k\in\Nz$. In the first place, we list some prerequisites for the {\em H\"older} and {\em little H\"older} spaces on $\X\in\{\R^m,\H\}$ from \cite[Section~11]{AmaAr}. 
Given any Banach space $F$,
the Banach space $BC^k(\X,F)$ is defined by
\begin{align*}
BC^k(\X,F):=(\{u\in C^k(\X,F):\|u\|_{k,\infty}<\infty \},\|\cdot\|_{k,\infty}).
\end{align*} 
The closed linear subspace $BU\!C^k(\X,F)$ of $BC^{k}(\X,F)$ consists of all functions $u\in BC^{k}(\X,F)$ such that $\partial^{\alpha}u$ is uniformly continuous for all $|\alpha|\leq k$. Moreover,
\begin{align*}
BC^{\infty}(\X,F):=\bigcap_{k}BC^{k}(\X,F)=\bigcap_{k}BU\!C^{k}(\X,F).
\end{align*}
It is a Fr\'echet space when equipped with the natural projective topology. 
\smallskip\\
For $0<s<1$, $0<\delta\leq\infty$ and $u\in F^{\X}$, the seminorm $[\cdot]^{\delta}_{s,\infty}$ is defined by
\begin{align*}
[u]^{\delta}_{s,\infty}:=\sup_{h\in(0,\delta)^m}\frac{\|u(\cdot+h)-u(\cdot)\|_{\infty}}{|h|^s}, \hspace*{1em}[\cdot]_{s,\infty}:=[\cdot]^{\infty}_{s,\infty}.
\end{align*}
Let $k<s<k+1$. The {\em H\"older} space $BC^{s}(\X,F)$ is defined as 
\begin{align*}
BC^{s}(\X,F):=(\{u\in BC^k(\X,F):\|u\|_{s,\infty}<\infty \},\|\cdot\|_{s,\infty}),
\end{align*}
where $\|u\|_{s,\infty}:=\|u\|_{k,\infty}+\max_{|\alpha|=k}[\partial^{\alpha} u]_{s-k,\infty}$.
\smallskip\\
The {\em little H\"older} space of order $s\geq 0$ is defined by
\begin{center}
$bc^s(\X,F):=$ the closure of $BC^{\infty}(\X,F)$ in $BC^{s}(\X,F)$.
\end{center}
By \cite[formula~(11.13), Corollary~11.2, Theorem~11.3]{AmaAr}, we have 
$$
bc^k(\X,F)=BU\!C^k(\X,F),
$$
and for $k<s<k+1$
\begin{center}
$u\in BC^s(\X,F)$ belongs to $bc^s(\X,F)$ iff $\lim\limits_{\delta\rightarrow 0}[\partial^{\alpha}u]^{\delta}_{s-[s],\infty}=0$, \hspace{1em} $|\alpha|=[s]$.
\end{center}
Now we are ready to introduce the weighted {\em H\"older} and {\em little H\"older} spaces on singular manifolds.
Define
$$BC^{k,\vartheta}(\M,V):=(\{u\in{C^k({\M},V)}:\|u\|_{k,\infty;\vartheta}<\infty\},\|\cdot\|_{k,\infty;\vartheta}),$$
where $\|u\|_{k,\infty;\vartheta}:={\max}_{0\leq{i}\leq{k}}\|\rho^{\vartheta+i+\tau-\sigma}|\nabla^{i}u|_{g}\|_{\infty}$. 
We also set
$$BC^{\infty,\vartheta}(\M,V):=\bigcap_{k}BC^{k,\vartheta}(\M,V)$$
endowed with the conventional projective topology. Then
$$
bc^{k,\vartheta}(\M,V):=\text{ the closure of }BC^{\infty,\vartheta}\text{ in }BC^{k,\vartheta}(\M,V).
$$
Let $k<s<k+1$. Now the {\em H\"older} space $BC^{s,\vartheta}(\M,V)$ is defined by
\begin{equation}
\label{S2.2: def-BC^s}
BC^{s,\vartheta}(\M,V):=(bc^{k,\vartheta}(\M,V),bc^{k+1,\vartheta}(\M,V))_{s-k,\infty}.
\end{equation}
Here $(\cdot,\cdot)_{\theta,\infty}$ is the real interpolation method, see \cite[Example I.2.4.1]{Ama95} and \cite[Definition~1.2.2]{Lunar95}. $BC^{s,\vartheta}(\M,V)$ equipped with the norm $\|\cdot \|_{s,\infty; \vartheta}$ is a Banach space by interpolation theory, where $\|\cdot \|_{s,\infty; \vartheta}$ is the norm of the interpolation space in definition~\eqref{S2.2: def-BC^s}. For $s\geq 0$, we define the weighted {\em little H\"older} spaces by 
\begin{equation}
\label{S2.2: def-bc^s}
bc^{s,\vartheta}(\M,V):=\text{ the closure of }BC^{\infty,\vartheta}(\M,V) \text{ in }BC^{s,\vartheta}(\M,V).
\end{equation}

\subsection{\bf Basic properties} 
In the following context, assume that $E_\kappa$ is a sequence of Banach spaces for $\kappa\in\K$. Then $\bE:=\prod_{\kappa}E_{\kappa}$. We denote by $l_\infty^\vartheta(\bE):=l_\infty^\vartheta(\bE;\rho)$ the linear subspace of $\bE$ consisting of all $\bu=(u_\kappa)$ such that
\begin{align*}
\|\bu\|_{l_\infty^\vartheta(\bE)}:=
\sup\limits_\kappa\|\rho_\kappa^\vartheta u_\kappa\|_{E_\kappa}<\infty.
\end{align*}
Then $l_\infty^\vartheta(\bE)$ is a Banach space with norm $\|\cdot\|_{l^\vartheta_\infty(\bE)}$. 
For $\F \in \{bc,BC\}$, we put $\bF^s:=\prod_{\kappa}{\F}^s_{\kappa}$, where ${\F}^s_{\kappa}:={\F}^s(\Xk,E)$. 
Denote by 
\begin{align*}
l^\vartheta_{\infty,\uf}(\boldsymbol{bc}^k)
\end{align*}
the linear subspace of $l^\vartheta_\infty(\boldsymbol{BC}^k)$ of all $\bu=(u_{\kappa})_{\kappa}$ such that $\rho_\kappa^\vartheta\partial^{\alpha}u_{\kappa}$ is uniformly continuous on $\Xk$ for $|\alpha|\leq k$, uniformly with respect to $\kappa\in\mathfrak{K}$. 
Similarly, for any $k<s<k+1$, we denote by 
\begin{align*}
l^\vartheta_{\infty,\uf}(\boldsymbol{bc}^s)
\end{align*}
the linear subspace of $l^\vartheta_{\infty,\uf}(\boldsymbol{bc}^k)$ of all $\bu=(u_{\kappa})_{\kappa}$ such that 
$$
\lim\limits_{\delta\rightarrow 0}\max_{|\alpha|=k}\rho_\kappa^\vartheta[\partial^{\alpha}u_{\kappa}]^{\delta}_{s-k,\infty}=0, 
$$
uniformly with respect to $\kappa\in\mathfrak{K}$. 

In the sequel, we always assume $\F\in\{bc,BC\}$, unless stated otherwise.
Define
$$L_\vartheta: l^{\vartheta^\prime+\vartheta}_b(\bF^s)\rightarrow l^{\vartheta^\prime}_b(\bF^s)): \quad (u_\kappa)_\kappa\mapsto (\rho_\kappa^\vartheta u_\kappa)_\kappa,$$ 
where $b=``\infty,\uf"$ for $\F=bc$, and $b=\infty$ for $\F=BC$.   Then we have the following proposition.
\begin{prop}
\label{S2: change of wgt-lq}
$L_\vartheta \in \Lis(l^{\vartheta^\prime+\vartheta}_b(\bF^s),l^{\vartheta^\prime}_b(\bF^s))$ with $(L_\vartheta)^{-1}=L_{-\vartheta}$.
\end{prop}
\begin{proof}
This follows immediately from the definition of weighted $l_b$ spaces.
\end{proof}

\begin{prop}
\label{S2: retraction&coretraction}
$\Re$ is a retraction from $l_b^\vartheta(\bF^s)$ onto $\F^{s,\vartheta}(\M,V)$ with $\Rc$ as a coretraction. Here $b=``\infty,\uf"$ for $\F=bc$, and $b=\infty$ for $\F=BC$. 
\end{prop}
\begin{proof}
In \cite{AmaAr}, a different retraction and coretraction system between $\F^{s,\vartheta}(\M,V)$ and $l_b(\bF^s)$ is defined as follows.
$$\mathcal{R}^{\vartheta; c}_{\infty,\kappa}:=\rho_{\kappa}^{\vartheta}\mathcal{R}^c_{\kappa},\quad \text{and}\quad \mathcal{R}^\vartheta_{\infty,\kappa}:=\rho_{\kappa}^{-\vartheta}\mathcal{R}_{\kappa};$$
and 
\begin{align*}
&\mathcal{R}^{\vartheta; c}_\infty:{L}_{1,loc}({\M},V)\rightarrow\boldsymbol{L}_{1,loc}({{\R}^{m}}),\quad u\mapsto{(\mathcal{R}^{\vartheta; c}_{\infty,\kappa}u)_{\kappa}},\\
\vspace{1em}
&\mathcal{R}^\vartheta_\infty:\boldsymbol{L}_{1,loc}({{\R}^{m}})\rightarrow{L}_{1,loc}({\M},V),\quad (v_{\kappa})_{\kappa}\mapsto{\sum_{\kappa}\mathcal{R}^\vartheta_{\infty,\kappa}v_{\kappa}}.
\end{align*}
We have the following relationship between these two retraction and coretraction systems:
$$\mathcal{R}^{\vartheta; c}_\infty=L_\vartheta\circ\Rc,\quad \mathcal{R}^\vartheta_\infty=\Re\circ L_{-\vartheta}. $$
Now the assertion follows straight away from Proposition~\ref{S2: change of wgt-lq} and \cite[Theorems~12.1, 12.3, formula~(12.7)]{AmaAr}. 
\end{proof}

In the sequel, $(\cdot,\cdot)^0_{\theta,\infty}$ and $[\cdot,\cdot]_\theta$ denote the continuous interpolation method and the complex interpolation method, respectively. See \cite[Example I.2.4.2, I.2.4.4]{Ama95} for definitions.
\begin{prop}
\label{S2: interpolation}
Suppose that $0<s_0<s_1<\infty$, $0<\theta<1$ and $\vartheta\in \R$. Then
$$(\F^{s_0,\vartheta}(\M,V), \F^{s_1,\vartheta}(\M,V))_\theta \doteq \F^{s_\theta,\vartheta}(\M,V)\doteq [\F^{s_0,\vartheta}(\M,V), \F^{s_1,\vartheta}(\M,V)]_\theta$$
holds for $s_0,s_1,s_\theta\notin \N$. When $\F=bc$,  $(\cdot,\cdot)_\theta=(\cdot,\cdot)^0_{\theta,\infty}$, and when $\F=BC$,  $(\cdot,\cdot)_\theta=(\cdot,\cdot)_{\theta,\infty}$. 
Here $\xi_\theta:=(1-\theta)\xi_0+\theta \xi_1$ for any $\xi_0,\xi_1\in \R$.
\end{prop}
\begin{proof}
See \cite[Corollaries~12.2, 12.4]{AmaAr}.
\end{proof}

\begin{prop}
\label{S2: interpolation-l_b}
Suppose that $0<s_0<s_1<\infty$, $0<\theta<1$ and $\vartheta \in \R$. Then
$$(l_b^\vartheta(\bF^{s_0}),l_b^\vartheta(\bF^{s_1}))_\theta\doteq 
l_b^\vartheta(\bF^{s_\theta}) \doteq 
[l_b^\vartheta(\bF^{s_0}),l_b^\vartheta(\bF^{s_1})]_\theta $$
holds for $s_0,s_1,s_\theta\notin \N$. When $\F=bc$, $b=``\infty,\uf"$ and $(\cdot,\cdot)_\theta=(\cdot,\cdot)^0_{\theta,\infty}$, and when $\F=BC$, $b=\infty$ and $(\cdot,\cdot)_\theta=(\cdot,\cdot)_{\theta,\infty}$.  
\end{prop}
\begin{proof}
The assertion with weight $\vartheta=0$ follows from \cite[Lemmas~11.10, 11.11]{AmaAr} and \cite[Proposition~I.2.3.2]{Ama95}. The remaining statement is a consequence of Proposition~\ref{S2: change of wgt-lq} and \cite[Proposition~I.2.3.2]{Ama95}.
\end{proof}


Let $V_j=V^{\sigma_j}_{\tau_j}:=\{T^{\sigma_j}_{\tau_j}\M,(\cdot|\cdot)_g\}$ with $j=1,2,3$ be $\bK$-valued tensor bundles on $\M$. Let $\oplus$ be the Whitney sum. By bundle multiplication from $V_1\times V_2$ into $V_3$, denoted by
\begin{center}
${\mathsf{m}}: V_1\oplus V_2\rightarrow V_3$,\hspace{1em} $(v_1,v_2)\mapsto {\mathsf{m}}(v_1,v_2)$,
\end{center}
we mean a smooth bounded section $\mathfrak{m}$ of $\Hom(V_1\otimes V_2,V_3)$, i.e., 
\begin{align}
\label{section 2: bundle multiplication}
\mathfrak{m}\in BC^{\infty}(\M, \text{Hom}(V_1\otimes V_2,V_3)), 
\end{align}
such that $\mathsf{m}(v_1,v_2):=\mathfrak{m}(v_1\otimes v_2)$. \eqref{section 2: bundle multiplication} implies that  for some $c>0$
\begin{center}
$|{\mathsf{m}}(v_1,v_2)|_g \leq c|v_1|_g |v_2|_g$,\hspace{1em} $v_i\in \Gamma(\M,V_i)$ with $i=1,2$.
\end{center}
Its point-wise extension from $\Gamma(\M,V_1\oplus V_2)$ into $\Gamma(\M,V_3)$ is defined by:
\begin{align*}
\mathsf{m}(v_1,v_2)(p):=\mathsf{m}(p)(v_1(p),v_2(p))
\end{align*}
for $v_i\in \Gamma(\M,V_i)$ and $p\in\M$. We still denote it by ${\mathsf{m}}$. We can  prove the following point-wise multiplier theorems for function spaces over singular manifolds.

\begin{prop}
\label{S2: pointwise multiplication}
Let $k\in\Nz$. Assume that the tensor bundles $V_j=V^{\sigma_j}_{\tau_j}:=\{T^{\sigma_j}_{\tau_j}\M,(\cdot|\cdot)_g\}$ with $j=1,2,3$ satisfy
\begin{align}
\label{section 2: ptm-condition}
\sigma_3-\tau_3=\sigma_1+\sigma_2-\tau_1-\tau_2.
\end{align}
Suppose that $\mathsf{m}:V_1\oplus V_2\rightarrow V_3$ is a bundle multiplication, and $\vartheta_3=\vartheta_1+\vartheta_2$. Then 
$$\F^{s,\vartheta_1}(\M,V_1)\times\F^{s,\vartheta_2}(\M,V_2)\rightarrow \F^{s,\vartheta_3}(\M,V_3),\quad [(v_1,v_2)\mapsto \mathsf{m}(v_1,v_2)]$$
 is a bilinear and continuous map.
\end{prop}
\begin{proof}
The statement follows from \cite[Theorem~13.5]{AmaAr}.
\end{proof}

\begin{prop}
\label{S2: change of wgt}
$$f_{\vartheta}:[u\mapsto \rho^{\vartheta}u] \in \Lis(\F^{s,\vartheta^\prime+\vartheta}(\M,V),\F^{s,\vartheta^\prime}(\M,V)),\quad (f_\vartheta)^{-1}=f_{-\vartheta}.$$
\end{prop}
\begin{proof}
By (S3) and (S4), we infer that $\boldsymbol{\rho}:=(\zeta\frac{\kf\rho^\vartheta}{\rho_\kappa^\vartheta})_{\kappa}\in \bigcap_k l_\infty(\boldsymbol{BC}^k)$, where $\zeta$ is defined in (L2). Then it follows from the point-wise multiplication results in \cite[Appendix~A2]{Ama01} and \cite[Corollary~2.8.2]{Trib83} that for $\bu=(u_\kappa)_\kappa$ and any $s\geq 0$
\begin{align*}
[\bu\mapsto (\zeta\frac{\kf\rho^\vartheta}{\rho_\kappa^\vartheta}u_\kappa)_\kappa] \in \L(l_\infty^{\vartheta^\prime}(\boldsymbol{BC}^s)).
\end{align*}
Given $u\in BC^{s,\vartheta^\prime}(\M,V)$,
\begin{align*}
\|\rho^\vartheta u\|_{s,\infty;\vartheta^\prime}&= 
\|\Re\Rc\rho^\vartheta u\|_{s,\infty;\vartheta^\prime}
\leq C\|\Rc \rho^\vartheta u\|_{l_\infty^{\vartheta^\prime}(\boldsymbol{BC}^s)}\\
&= \|\boldsymbol{\rho} L_{\vartheta}\Rc u\|_{l_\infty^{\vartheta^\prime}(\boldsymbol{BC}^s)}
\leq  C\|\boldsymbol{\rho}\|_{l_\infty(\boldsymbol{BC}^k)}\|\Rc u\|_{l_\infty^{\vartheta^\prime+\vartheta}(\boldsymbol{BC}^s)}\\
&\leq C(\rho,\vartheta,k)\|u\|_{s,\infty;\vartheta^\prime+\vartheta}.
\end{align*}
Now the open mapping theorem implies that the asserted result for $\F=BC$. Given any $u\in bc^{s,\vartheta^\prime+\vartheta}(\M,V)$, then there exists $(u_n)_n \in BC^{\infty, \vartheta^\prime+\vartheta}(\M,V)$ such that $u_n\to u$ in $BC^{s, \vartheta^\prime+\vartheta}(\M,V)$. We already have 
$$\| \rho^\vartheta u\|_{s,\infty; \vartheta^\prime} \leq C \|u\|_{s,\infty; \vartheta^\prime+\vartheta} ,$$
and $(\rho^\vartheta u_n)_n \in BC^{\infty, \vartheta^\prime}(\M,V)$. By the conclusion for $\F^s=BC^s$, we infer that as $n\to\infty$
$$\| \rho^\vartheta (u-u_n)\|_{s,\infty; \vartheta^\prime} \leq C \|u-u_n\|_{s,\infty; \vartheta^\prime+\vartheta} \to 0 .$$
We have established the asserted result for weighted {\em little H\"older} spaces in view of the  definition~\eqref{S2.2: def-bc^s}. 
\end{proof}

\begin{prop}
\label{S2: nabla}
For any $\sigma,\tau\in\Nz$ and $\vartheta\in\R$,
$$\nabla\in \L(\F^{s,\vartheta}(\M,V^\sigma_\tau), \F^{s-1,\vartheta}(\M,V^\sigma_{\tau+1})).$$
\end{prop}
\begin{proof}
The case $s\in\N$ is immediate from the definition of the weighted function spaces. The non-integer case follows from \cite[Theorem~16.1]{AmaAr}.
\end{proof}

Let $\hat{g}=g/\rho^2$. 
Then $(\M,\hat{g})$ is a uniformly regular Riemannian manifold. We denote the corresponding tensor fields by $\hat{V}=\hat{V}^\sigma_\tau$. 
The definitions of the corresponding weighted function spaces $\F^{s^\prime,\vartheta^\prime}(\M,\hat{V})$ do not depend on the choice of $\vartheta^\prime$ in this case.
We denote the unweighted spaces by $\F^{s^\prime}(\M,\hat{V})$. 
The reader may refer to \cite{Shao13} for the precise definitions for these unweighted spaces on uniformly regular Rimannian manifolds. 
\begin{prop}
\label{S2.3: unweighted spaces}
For $\F\in \{bc, BC, W_p, \mathring{W}_p\}$, it holds that
$$\F^s(\M,\hat{V})\doteq \F^{s,-1/p}(\M,V) $$
\end{prop}
\begin{proof}
The assertion follows from Proposition~\ref{S2: retraction&coretraction} and \cite[Propositions~2.1, 2.2]{Shao13}.
\end{proof}

\section{\bf Continuous maximal regularity}

\subsection{\bf Continuous maximal regularity on singular manifolds}
Throughout the rest of this paper, we always assume that $(\M,g;\rho)$ is a singular manifold without boundary.

Following \cite{ShaoSim13}, letting	$l\in\Nz$, $\cA:C^{\infty}({\M},V)\rightarrow \Gamma({\M},V)$ is called a linear differential operator of order $l$ on $\M$ if we can find $\bfa=(a_r)_r\in \prod_{r=0}^l \Gamma(\M, V^{\sigma+\tau+r}_{\tau+\sigma})$ such that
\begin{align}
\label{S3: globally-defined diff-op}
\cA=\cA(\bfa):=\sum\limits_{r=0}^l \ev(a_r,\nabla^r \cdot).
\end{align}
Here complete contraction 
\begin{align*}
\ev:\Gamma(\M, V^{\sigma+\tau+r}_{\tau+\sigma}\times V^{\sigma}_{\tau+r})\rightarrow \Gamma(\M, V^\sigma_\tau): (a,b)\mapsto \ev(a,b)
\end{align*}
is defined as follows. Let $(i_1),(i_2),(i_3)\in\J^{\sigma}$, $(j_1),(j_2),(j_3)\in\J^{\tau}$ and $(k_1),(k_2)\in\J^{r}$. Then
\begin{align*}
&\ev(a,b)(\p):=\ev(a^{(i_3;j_1;k_1)}_{(j_3;i_1)} \frac{\partial}{\partial x^{(i_3)}}\otimes \frac{\partial}{\partial x^{(j_1)}}\otimes\frac{\partial}{\partial x^{(k_1)}}\otimes dx^{(j_3)}\otimes dx^{(i_1)},\\
&\quad  b^{(i_2)}_{(j_2;k_2)} \frac{\partial}{\partial x^{(i_2)}}\otimes dx^{(j_2)}\otimes dx^{(k_2)})(\p)\\
&=a^{(i_3;j_1;k_1)}_{(j_3;i_1)} b^{(i_1)}_{(j_1;k_1)}\frac{\partial}{\partial x^{(i_3)}}\otimes dx^{(j_3)}(\p),
\end{align*}
in every local chart and for $p\in\M$. The index $(i_3;j_1;k_1)$ is defined by
\begin{align*}
(i_3;j_1;k_1)=(i_{3,1},\cdots,i_{3,\sigma};j_{1,1},\cdots,j_{1,\tau};k_{1,1},\cdots,k_{1,r}).
\end{align*}
The other indices are defined in a similar way. \cite[Lemma~14.2]{AmaAr} implies that $\ev$ is a bundle multiplication. Making use of \cite[formula~(3.18)]{Ama13}, one can check that for any $l$-th order linear differential operator so defined, in every local chart $(\Ok,\vpk)$ there exists some linear differential operator
\begin{align}
\label{S3: local exp of diff-op}
\cA_{\kappa}(x,\partial):=\sum\limits_{|\alpha|\leq{l}}a^{\kappa}_{\alpha}(x)\partial^{\alpha}, \hspace*{.5em}\text{ with }a^{\kappa}_{\alpha}\in \L(E)^{\Qk},
\end{align}
called the local representation of $\cA$ in $(\Ok,\vpk)$, such that for any $u\in C^{\infty}({\M},V)$
$$
\kf(\cA u)=\cA_{\kappa}(\kf u).
$$

\begin{prop}
\label{S3: diff-op}
Let $s\geq 0$ and $\vartheta \in\R$.
Suppose that $\cA=\cA(\bfa)$ with $\bfa=(a_r)_r\in \prod_{r=0}^l bc^s(\M,V^{\sigma+\tau+r}_{\tau+\sigma})$.  Then 
\begin{align*}
\cA\in\L(\F^{s+l,\vartheta}(\M,V),\F^{s,\vartheta}(\M,V)). 
\end{align*}
\end{prop}
\begin{proof}
The assertion is a direct consequence of Propositions~\ref{S2: pointwise multiplication} and \ref{S2: nabla}.
\end{proof}

Given any angle $\phi\in [0,\pi]$, set 
$$\Sigma_\phi:=\{z\in\C: |{\rm arg}z|\leq \phi \}\cup\{0\}.$$
A linear operator $\cA:=\cA(\bfa)$ of order $l$ is said to be {\em normally $\rho$-elliptic} if there exists some constant $\Ce>0$ such that for every pair $(\p,\xi)\in \M\times\Gamma(\M, T^\ast M)$ with $|\xi(\p)|_{g^*(\p)}\neq 0$ for all $\p\in\M$, the principal symbol 
$$\hat{\sigma}\cA^\pi(\p,\xi(\p)):=\ev(a_l,(-i\xi)^{\otimes l})(\p)\in \L(T_\p\M^{\otimes\sigma}\otimes T_\p^*\M^{\otimes\tau})$$ 
satisfies 
\begin{align}
\label{S3: nor rho-ept-1}
S:=\Sigma_{\pi/2} \subset \rho(-\hat{\sigma}\cA^\pi(\p,\xi(\p))),
\end{align}
and
\begin{align}
\label{S3: nor rho-ept-2}
(\rho^l(\p)|\xi(\p)|_{g^*(\p)}^l +|\mu|) \|(\mu + \hat{\sigma}\cA^\pi(\p,\xi(\p)))^{-1}\|_{\L(T_p\M^{\otimes\sigma}\otimes T_p^*\M^{\otimes\tau})} \leq \Ce ,\quad \mu\in S. 
\end{align}
The constant $\Ce$ is called the \emph{$\rho$-ellipticity constant} of $\cA$. To the best of the author's knowledge, this ellipticity condition is the first one formulated for degenerate or singular elliptic operators acting on tensor fields.

We can also introduce a stronger version of the ellipticity condition for $\cA$. $\cA$ is called 
{\em uniformly strongly} $\rho$-{\em elliptic} if there exists some constant $C_e>0$ such that for all $(\p,\xi,\eta)\in \M\times\Gamma(\M, T^\ast M)\times \Gamma(\M,T^\sigma_\tau \M)$  the principal symbol satisfies 
$$
\hat{\sigma}\cA^\pi(\p,\xi(\p))(\eta(\p))\geq C_e \rho^l(\p) |\eta(\p)|_{g(\p)}^2 |\xi(\p)|^l_{g^*(\p)}. 
$$
Here $\hat{\sigma}\cA^\pi(\p,\xi(\p))(\eta(\p)):=(\ev(a_l,\eta\otimes(-i\xi)^{\otimes l})(\p)|\eta(\p))_{g(\p)}$. 
In \cite{Ama13b}, H.~Amann has used the {\em uniformly strong $\rho$-ellipticity} condition to establish an $L_p$-maximal regularity theory for second order differential operators acting on scalar functions. 

We can readily check that a {\em uniformly strongly} $\rho$-{\em elliptic} operator $\cA$ must be {\em normally $\rho$-elliptic}.
If $\cA$ is of odd order, then by replacing $\xi$ with $-\xi$ in \eqref{S3: nor rho-ept-1}, it is easy to see that $\rho(\hat{\sigma}\cA^\pi(\p,\xi(\p)))=\C$. This is a contradiction. Therefore, every {\em normally $\rho$-elliptic} operator is of even order.

We call a linear operator $\cA:=\cA(\bfa)$ $s$-{\em regular} if 
\begin{align}
\label{S3: rho-reg}
a_r\in bc^s(\M, V^{\sigma+\tau+r}_{\tau+\sigma}), \quad r=0,1,\cdots,l.
\end{align}
This reveals the existence of some constant $\Ca$ such that 
\begin{align}
\label{S3: rho-reg-2}
\|a_r\|_{s,\infty} \leq \Ca, \quad r=0,1,\cdots,l.
\end{align}
We consider how \eqref{S3: rho-reg} affects the behavior of the localizations $\cA_{\kappa}$. 
Given any linear differential operator $\cA$ of order $2l$,
by an analogy of Proposition~\ref{S2: retraction&coretraction}, we infer that
$$(\kf a_r)_\kappa \in l_{\infty,\uf}(\boldsymbol{bc}^s(\Qk, E^{\sigma+\tau+r}_{\tau+\sigma})), \quad r=0,1,\cdots,2l,$$
or equivalently
$$(\kf (a_r)^{(i)}_{(j)})_\kappa \in l_{\infty,\uf}(\boldsymbol{bc}^s(\Qk)), \quad (i)\in\J^{\sigma+\tau+r},\quad (j)\in \J^{\tau+\sigma},\quad r=0,1,\cdots,2l.$$ 
By \cite[formula~(3.18)]{Ama13}, the coefficients of $\cA_{\kappa}$, i.e., $a^\kappa_\alpha$, are linear combinations of the products of $(a_r)^{(i)}_{(j)}$ and possibly the derivatives of the Christoffel symbols of the metric $g$. Thus \cite[formula~(3.19)]{Ama13} shows that
\begin{align}
\label{S3: rho-reg-loc}
(a^\kappa_\alpha)_\kappa \in l_{\infty,\uf}(\boldsymbol{bc}^s(\Qk,\L(E))) ,\quad |\alpha|\leq 2l.
\end{align}

Given any Banach space $X$, a linear differential operator of order $l$ 
$$\cA:=\cA(x,\partial):=\sum\limits_{|\alpha|\leq l} a_{\alpha}(x)\partial^{\alpha}$$ 
defined on an open subset $U\subset\R^m$ with $a_\alpha: U \rightarrow \L(X)$ is said to be {\em normally elliptic} if its principal symbol
$\hat{\sigma}\cA^\pi(x,\xi):=\sum\limits_{|\alpha|=l}a_{\alpha}(x)(-i\xi)^{\alpha}$
satisfies 
$$S:=\Sigma_{\pi/2}\subset \rho(-\hat{\sigma}\cA^\pi(x,\xi))$$ 
and there exists some $\Ce>0$ such that
\begin{align}
\label{S3:: nor ept-loc}
(|\xi|^l+|\mu|)\| (\mu+ \hat{\sigma}\cA^{\pi}(x,\xi))^{-1}\|_{\L(X)}\leq \Ce, \quad \mu\in S,
\end{align}
for all $(x,\xi)\in U\times \dot{\R}^m$, where $\dot{\R}^m:=\R^m\setminus\{0\}$. The constant $\Ce$ is called the \emph{ellipticity constant} of $\cA$. As above, one can check that  $\cA$ must be of even order.

\begin{prop}
\label{S3: equiv-ellip}
A linear differential operator $\cA:=\cA(\bfa)$ of order $2l$ is normally $\rho$-elliptic iff all its local realizations 
$$
\cA_{\kappa}(x,\partial)=\sum\limits_{|\alpha|\leq{2l}}a^{\kappa}_{\alpha}(x)\partial^{\alpha}
$$
are  normally elliptic on $\Qk$ with a uniform {\em ellipticity constant} $\Ce$ in condition~\eqref{S3:: nor ept-loc}.
\end{prop}
\begin{proof}
We first assume that $\cA:=\cA(\bfa)$ is {\em normally $\rho$-elliptic}. In every local chart $(\Ok,\vpk)$, by definition we have
$$
\hat{\sigma}\cA_\kappa^\pi(x,\xi)=\sum\limits_{|\alpha|=2l}a^\kappa_\alpha(x) (-i\xi)^\alpha= \kf \ev(a_{2l},(-i\xi^\M)^{\otimes 2l})(\p)
$$
with $(x,\xi)\in\Qk\times\dot{\R}^m$ and $\p=\psk(x)$. 
Here $\xi^\M$ is a $1$-form satisfying $\xi^\M|_{\Ok}=\xi_j dx^j$.
By \cite[formula~(3.2)]{ShaoSim13} and \eqref{S3: nor rho-ept-1}, we conclude $S:=\Sigma_{\pi/2}\subset \rho(-\hat{\sigma}\cA_\kappa^\pi(x,\xi))$. For every $\mu\in S$, $\eta,\varsigma\in E^\sigma_\tau$ with $\varsigma=(\mu+\hat{\sigma}\cA_\kappa^\pi(x,\xi))\eta$, and $\xi\in\dot{\R}^m$, one computes
\begin{align}
\notag &\quad  (|\xi|^{2l}_{g_m} +|\mu|) | (\mu + \hat{\sigma}\cA^{\pi}(x,\xi))^{-1} \varsigma |_{g_m}=(|\xi|^{2l}_{g_m} +|\mu|) | \eta |_{g_m}\\
\label{S3.1: eq1}
&\leq C \rho_\kappa^{\tau-\sigma}(C^\prime \rho^{2l}(\p) |\xi^\M (\p)|^{2l}_{g^*(\p)} +|\mu|) | d\psk (x)\eta |_{g(\p)}\\
\label{S3.1: eq2}
& \leq M  \rho_\kappa^{\tau-\sigma}(\rho^{2l}(\p) |\xi^\M (\p)|^{2l}_{g^*(\p)} +|\mu|) | d\psk (x)\eta |_{g(\p)}\\
\label{S3.1: eq3}
& \leq M \Ce  \rho_\kappa^{\tau-\sigma}| (\mu + \ev(a_{2l},(-i\xi^\M)^{\otimes 2l})(\p)) d\psk (x)\eta |_{g(\p)}\\
\label{S3.1: eq4}
&\leq M^\prime \Ce  \rho_\kappa^{\tau-\sigma} \rho_\kappa^{\sigma-\tau}|\kf (\mu + \ev(a_{2l},(-i\xi^\M)^{\otimes 2l})(\p)) d\psk (x)\eta |_{g_m}\\
\notag &= M^\prime \Ce  | (\mu +\hat{\sigma}\cA_\kappa^\pi(x,\xi)) \eta |_{g_m} =M^\prime\Ce |\varsigma|_{g_m}.
\end{align}
In \eqref{S3.1: eq1}, we have adopted (S4) and (P3). In \eqref{S3.1: eq2}, the constant $M=C\max\{C^\prime,1\}$ is independent of the choices of $\kappa$ and $x$.  \eqref{S3.1: eq3}  follows from \eqref{S3: nor rho-ept-2}, and \eqref{S3.1: eq4} is a direct consequence of (P3).

The ``if" part follows by a similar argument.
\end{proof}

\begin{prop}
\label{S3: Prop-res-est}
Let $s\in \R_+\setminus\N$ and $\vartheta\in \R$. Suppose that $\cA=\cA(\bfa)$ is a $2l$-th order linear differential operator, which is normally $\rho$-elliptic and $s$-regular with bounds $\Ce$ and $\Ca$ defined in \eqref{S3: nor rho-ept-2} and \eqref{S3: rho-reg-2}. Then  there exist $\omega=\omega(\Ce,\Ca)$, $\phi=\phi(\Ce,\Ca)>\pi/2$ and $\cE=\cE(\Ce,\Ca)$ such that $S=\omega+ \Sigma_\phi\subset \rho(-\cA)$ and 
$$|\mu|^{1-i} \|(\mu+\cA)^{-1}\|_{\L(\F^{s,\vartheta}(\M,V), \F^{s+2li,\vartheta}(\M,V))}\leq \cE, \quad \mu\in S,\quad i=0,1.$$
\end{prop}
\begin{proof}
To economize notation, we set 
$$
E_0:=\F^{s,\vartheta},\quad E_\theta:=\F^{s+2l-1,\vartheta} ,\quad E_1:=\F^{s+2l,\vartheta},
$$
and
$$
l_b^\vartheta(\bE_0):=l_b^{\vartheta}(\bF^s),\quad  l_b^\vartheta(\bE_\theta):=l_b^{\vartheta}(\bF^{s+2l-1}),\quad l_b^\vartheta(\bE_1):=l_b^{\vartheta}(\bF^{s+2l}),
$$
where $b=``\infty,\uf"$ for $\F=bc$, and $b=\infty$ for $\F=BC$.

(i) Define $h:\R^m\rightarrow \Q$: $x\mapsto\zeta(x)x$. Here $\zeta$ is defined in (L2). It is easy to see that $h\in BC^\infty(\R^m,\Q)$. Let
$$
\Ak(x,\partial):=\sum\limits_{|\alpha|\leq{2l}}\bar{a}_{\alpha}^{\kappa}(x)\partial^{\alpha}:=\sum\limits_{|\alpha|\leq{2l}}(a_{\alpha}^{\kappa}\circ h)(x)\partial^{\alpha}.
$$
It is not hard to check with the assistance of \eqref{S3: rho-reg-loc}  that the coefficients $(\bar{a}_\alpha^\kappa)_\kappa$ satisfy
$$
( \bar{a}^\kappa_\alpha)_\kappa \in l_{\infty,\uf}(\boldsymbol{bc}^s(\L(E))) ,\quad |\alpha|\leq 2l,
$$
and by Proposition~\ref{S3: equiv-ellip} that 
$\Ak$ are all {\em normally elliptic} with a uniform {\em ellipticity constant} for all $\kappa\in\K$.
In virtue of \cite[Theorems~4.1, 4.2 and Remark~4.6]{Ama01}, these two conditions imply the existence of some constants $\omega_0=\omega_0(\Ce,\Ca)$, $\phi=\phi(\Ce,\Ca)>\pi/2$ and $\cE=\cE(\Ce,\Ca)$ such that 
\begin{align}
\label{S3: res-Awkl}
S_0:=\omega_0 +\Sigma_\phi\subset \rho(-\Ak),\quad \kappa\in\K,
\end{align}
and
\begin{align}
\label{S3: res-est-Awkl}
|\mu|^{1-i} \|(\mu+\Ak)^{-1}\|_{\L(\F^s(E), \F^{s+2li}(E))}\leq \cE, \quad \mu\in S_0,\quad i=0,1, \quad \kappa\in\K.
\end{align}
Let $\bA:l_b^\vartheta(\bE_1)\rightarrow \bE$: $[(u_\kappa)_\kappa \mapsto (\Ak u_\kappa)_\kappa]$. 
First, it is not hard to verify by means of the point-wise multiplication results in \cite[Appendix~A2]{Ama01} that
\begin{equation}
\label{S3.1: bA is bdd-BC}
\bA\in\L(l_b^\vartheta(\bE_1), l_\infty^\vartheta(\bE_0)).
\end{equation}
By Proposition~\ref{S2: interpolation-l_b} and the well-known interpolation theory, for any $s<t\notin\N$, 
$$l^{\vartheta}_{\infty,\uf}(\boldsymbol{bc}^{t+2l})\overset{d}{\hookrightarrow} l^{\vartheta}_{\infty,\uf}(\boldsymbol{bc}^{s+2l}). $$
Hence for any $\bu\in l^\vartheta_{\infty,\uf}(\boldsymbol{bc}^{s+2l})$, we can choose 
$$(\boldsymbol{u}_n)_n:=((u_{n,\kappa})_\kappa)_n\subset l^{\vartheta}_{\infty,\uf}(\boldsymbol{bc}^{t+2l})$$ converging to $\bu$ in $l^{\vartheta}_\infty(\boldsymbol{bc}^{s+2l})$. Since $s$ is arbitrary, we see that the estimate \eqref{S3.1: bA is bdd-BC} still holds when $s$ is replaced by $t$, i.e., 
$$\bA\boldsymbol{u}_n\in l^{\vartheta}_\infty(\boldsymbol{bc}^t)\hookrightarrow l^{\vartheta}_{\infty,\uf}(\boldsymbol{bc}^s).$$
What is more, $\bA\boldsymbol{u}_n=: \boldsymbol{v}_n \to \bA\bu$ in the $l^{\vartheta}_\infty(\boldsymbol{BC}^s)$-norm. Since $l^{\vartheta}_{\infty,\uf}(\boldsymbol{bc}^{s+2l})$ is a Banach space, it yields
$
\bA\bu\in l_{\infty,\uf}^{\vartheta}(\boldsymbol{bc}^s). 
$
Therefore
\begin{equation}
\label{S3.1: bA is bdd}
\bA\in\L(l_b^\vartheta(\bE_1), l_b^\vartheta(\bE_0)).
\end{equation}
For any $\mu\in S_0$, it is easy to see that $\mu+\bA: \bF^{s+2l}\rightarrow l_\infty^\vartheta(\bE_0)$ is a bijective map. We write the inverse of $\mu+\bA$ as $(\mu+\bA)^{-1}$ and compute for $\bu:=(u_\kappa)_\kappa \in l_b^\vartheta(\bE_0)$
\begin{align}
\label{S3: res-est-bAw-1}
\notag\|(\mu+\bA)^{-1} \bu\|_{l_\infty^\vartheta(\boldsymbol{BC}^{s+2l})}&=\sup\limits_{\kappa\in\K} \rho_\kappa^\vartheta\| (\mu+\Ak)^{-1}  u_\kappa\|_{s+2l,\infty}\\
\notag &= \sup\limits_{\kappa\in\K} \| (\mu+\Ak)^{-1} \rho_\kappa^\vartheta u_\kappa\|_{s+2l,\infty}\\
&\leq \cE \sup\limits_{\kappa\in\K}  \|\rho_\kappa^\vartheta u_\kappa\|_{\F^s(E)}=\cE\|\bu\|_{l_b^\vartheta(\bE_0)}.
\end{align}
In the case $\F=bc$, \eqref{S3: res-est-bAw-1} only shows that for each $\bu\in l^\vartheta_{\infty,\uf}(\boldsymbol{bc}^s)$ and $\mu\in S$ $(\mu+\bA)^{-1}\bu\in l^\vartheta_\infty(\boldsymbol{BC}^{s+2l})$. It remains to prove $(\mu+\bA)^{-1}\bu\in l^\vartheta_{\infty,\uf}(\bE_1)$. This can be answered by a density argument as in the proof for \eqref{S3.1: bA is bdd}. 

Hence $S_0\subset \rho(-\bA)$. Similarly, one checks
\begin{align*}
|\mu|\|(\mu+\bA)^{-1} \|_{\L(l_b^\vartheta(\bE_0))}\leq \cE , \quad \mu\in S_0. 
\end{align*} 
\smallskip\\
(ii) Given any $u\in E_1(\M,V)$ and $\mu\in S$, one computes
\begin{align*}
&\quad[\Rck(\mu+\cA)-(\mu+\Ak)\Rck] u\\
&=\kf (\pk(\mu+\cA)u)-(\mu+\Ak)\kf(\pk u)\\
&=\kf \pk(\mu+\Ak)\kf u-(\mu+\Ak)\kf(\pk u) \\
&=\kf \pk\Ak\kf u-\Ak\kf(\pk u)\\
&=-\sum\limits_{|\alpha|\leq{2l}}\sum\limits_{0<\beta\leq\alpha}\binom{\alpha}{\beta}\bar{a}_{\alpha}^{\kappa}\partial^{\alpha-\beta}(\zeta\kf u)\partial^{\beta}(\kf\pk)=:\Bk u.
\end{align*}
Note that $\zeta\equiv 1$ on $\supp(\kf \pk)$ for all $\kappa\in\K$.
Define for any $u\in C^\infty(\M,V)$
$$\cB u:= (\Bk u)_\kappa.$$
Similar to the computation for \eqref{S3.1: bA is bdd}, we can easily check 
$$\cB\Re\in \L(l_b^\vartheta(\bE_\theta),l_b^\vartheta(\bE_0)). $$
By Proposition~\ref{S2: interpolation-l_b}, we have
$$l_b^\vartheta(\bE_\theta)\doteq (l_b^\vartheta(\bE_0),l_b^\vartheta(\bE_1))_\theta, $$
where either $(\cdot,\cdot)_{\theta}=(\cdot,\cdot)^0_{\theta,\infty}$ for $\F=bc$, or $(\cdot,\cdot)_{\theta}=(\cdot,\cdot)_{\theta,\infty}$ for $\F=BC$, and $\theta=1-1/(2l)$. 
\smallskip\\
It follows from interpolation theory and Proposition~\ref{S2: change of wgt-lq} that for every $\varepsilon>0$ there exists some positive constant $C(\varepsilon)$ such that for all $\bu\in l_b^\vartheta(\bE_1)$
$$
\|\cB\Re \bu\|_{l_b^\vartheta(\bE_0)} \leq \varepsilon \|\bu\|_{l_b^\vartheta(\bE_1)} + C(\varepsilon)\|\bu\|_{l_b^\vartheta(\bE_0)}
$$
Given any $\bu\in l_b^\vartheta(\bE_0)$ and  $\mu\in S_0$,
\begin{align*}
\|\cB\Re (\mu +\bA)^{-1}\bu\|_{l_b^\vartheta(\bE_0)}
\leq & \varepsilon \|(\mu +\bA)^{-1}\bu\|_{l_b^\vartheta(\bE_1)} + C(\varepsilon) \|(\mu +\bA)^{-1}\bu\|_{l_b^\vartheta(\bE_0)}\\
\leq & \cE(\varepsilon   + \frac{C(\varepsilon)}{|\mu|}) \|\bu\|_{l_b^\vartheta(\bE_0)}.
\end{align*}
Hence we can find some $\omega_1=\omega_1(\Ce,\Ca)\geq \omega_0$ such that for all $\mu\in S_1:=\omega_1 +\Sigma_\phi$
$$\|\cB\Re (\mu +\bA)^{-1}\|_{\L(l_b^\vartheta(\bE_0))} \leq 1/2, $$
which implies that $S_1 \subset \rho(-\bA-\cB\Re)$ and 
$$ \|(I +\cB\Re(\mu  +\bA)^{-1})^{-1}\|_{l_b^\vartheta(\bE_0)} \leq 2 .$$
Now we compute for any $\bu\in l_b^\vartheta(\bE_0)$ and $\mu\in S_1$
\begin{align*}
 |\mu| \|(\mu +\bA+\cB\Re)^{-1} \bu \|_{l_b^\vartheta(\bE_0)}
=& |\mu| \|(\mu  +\bA)^{-1} (I +\cB\Re(\mu  +\bA)^{-1})^{-1} \bu \|_{l_b^\vartheta(\bE_0)}\\
\leq & \cE \|(I +\cB\Re(\mu  +\bA)^{-1})^{-1} \bu \|_{l_b^\vartheta(\bE_0)}\\
\leq &  2\cE \|\bu\|_{l_b^\vartheta(\bE_0)},
\end{align*}
where $I=\id_{l_b^\vartheta(\bE_0)}$, 
and a similar computation yields
$$\|(\mu +\bA+\cB\Re)^{-1} \bu \|_{l_b^\vartheta(\bE_1)}\leq 2\cE \|\bu\|_{l_b^\vartheta(\bE_0)}.
$$
One readily checks 
$$\Rc (\mu+\cA)u =(\mu+\bA)\Rc u +B\Re\Rc u = (\mu +\bA+\cB\Re)\Rc u.$$
For $\mu\in S_1$, we immediately have
$$\Re(\mu +\bA+\cB\Re)^{-1}\Rc (\mu+\cA)=\Re (\mu +\bA+\cB\Re)^{-1} (\mu +\bA+\cB\Re)\Rc=\id_{E_1(\M,V)}. $$
Therefore, $\mu+\cA$ is injective for $\mu\in S_1$.

(iii) Given $u\in C^\infty(E):=C^\infty(\R^m, E)$, we define
$$\Ck u:=[(\mu +\cA)\Rek- \Rek (\mu+ \Ak)]u.$$
An easy computation shows that for each $u\in C^\infty(E)$
\begin{align*}
\kf \Ck u&= \sum\limits_{|\alpha|\leq 2l} \bar{a}^\kappa_\alpha \partial^\alpha (\kf \pk u) - \kf\pk (\sum\limits_{|\alpha|\leq 2l} \bar{a}^\kappa_\alpha \partial^\alpha u )\\
&= \sum\limits_{|\alpha|\leq 2l} \sum\limits_{0<\beta\leq \alpha} \binom{\alpha}{\beta} \bar{a}^\kappa_\alpha \partial^{\alpha-\beta}(\zeta u) \partial^\beta (\kf \pk).
\end{align*}
It is obvious that $\Ck\in \L(\F^{s+2l-1}(E),\F^s(\M,V))$. Moreover, with $\bu=(u_\kappa)_\kappa$, it is a simple matter to verify as for \eqref{S3.1: bA is bdd} that
$$[\bu \mapsto (\kf \Ck u_\kappa)_\kappa]\in \L(l_b^\vartheta(\bE_\theta), l_b^\vartheta(\bE_0)). $$
Define $\cC: l_b^\vartheta(\bE_\theta)\rightarrow E_1(\M,V)$: $[\bu \mapsto \sum\limits_\kappa \Ck u_\kappa]$. Then given any $\bu\in l_b^\vartheta(\bE_1)$
$$(\mu+\cA)\Re \bu= \Re(\mu+\bA)\bu + \Re\Rc\cC\bu=\Re(\mu+\bA+\Rc\cC)\bu. $$
It follows in an analogous way to the proof for Proposition~\ref{S2: retraction&coretraction} that 
$$[\bu\mapsto \sum\limits_\kappa \kb(\zeta u_\kappa)]\in \L(l_b^\vartheta(\bE_0), E_0(\M,V)).$$
In view of $\cC \bu= \sum\limits_\kappa \kb (\zeta\kf\Ck u_\kappa)$, we obtain
$$\cC \in \L(l_b^\vartheta(\bE_\theta),E_0(\M,V)) $$
and thus
$$\Rc\cC\in  \L(l_b^\vartheta(\bE_\theta), l_b^\vartheta(\bE_0)).$$
Now it is not hard to verify via an analogous computation as in (ii) that there exists some $\omega_2=\omega_2(\Ce,\Ca)\geq \omega_1$ such that $S_2:=\omega_2 +\Sigma_\phi\subset \rho(-\bA-\Rc\cC)$ and 
$$|\mu|^{1-i} \|(\mu+\bA+\Rc\cC)^{-1} \|_{\L(l_b^\vartheta(\bE_0),l_b^\vartheta(\bE_i))}\leq 2\cE, \quad \mu\in S_2,\quad i=0,1.$$
Then we have
$$(\mu+\cA)\Re(\mu+\bA+\Rc\cC)^{-1}\Rc=\Re(\mu+\bA+\Rc\cC)(\mu+\bA+\Rc\cC)^{-1}\Rc=\id_{E_0(\M,V)} .$$
Thus, $\mu+\cA$ is surjective for $\mu\in S_1$, and $\Re(\mu+\bA+\Rc\cC)^{-1}\Rc$ is a right inverse of $(\mu+\cA)$.  Furthermore,
\begin{align*}
&|\mu|^{1-i} \|(\mu+\cA)^{-1}\|_{\L(E_0(\M,V),E_i(\M,V))}\\
 =& |\mu|^{1-i} \|\Re(\mu+\bA+\Rc\cC)^{-1}\Rc\|_{\L(E_0(\M,V),E_i(\M,V))}\leq C\cE ,\quad\mu\in S_1,\quad i=0,1.
\end{align*}
This completes the proof
\end{proof}

Recall that an operator $A$ is said to belong to the class $\cH(E_1,E_0)$ for some densely embedded Banach couple $E_1\overset{d}{\hookrightarrow}E_0$, if $-A$ generates a strongly continuous analytic semigroup on $E_0$ with $dom(-A)=E_1$. By the well-known semigroup theory, Proposition~\ref{S3: Prop-res-est} immediately implies
\begin{theorem}
\label{S3: analytic semigroup}
Let $s\in\R_+\setminus\N$ and $\vartheta\in \R$. Suppose $\cA$ satisfies the conditions in Proposition~\ref{S3: Prop-res-est}. Then
$$
\cA\in\cH(bc^{s+2l,\vartheta}(\M,V),bc^{s,\vartheta}(\M,V)).
$$
\end{theorem}

For some fixed interval $I=[0,T]$, $\gamma\in(0,1)$, and some Banach space $X$, we define
\begin{align*}
&BU\!C_{1-\gamma}(I,X):=\{u\in{C(\dot{I},X)};[t\mapsto{t^{1-\gamma}}u]\in C(\dot{I},X),\lim\limits_{t\to{0^+}} t^{1-\gamma}\|u(t)\|_X=0\},\\
& \|u\|_{C_{1-\gamma}}:=\sup_{t\in{\dot{I}}}{t^{1-\gamma}}\|u(t)\|_{X},
\end{align*}
and
$$BU\!C_{1-\gamma}^1(I,X):=\{u\in{C^1(\dot{I},X)}: u,\dot{u}\in{BU\!C_{1-\gamma}(I,X)}\}.$$
Recall that in the above definition $\dot{I}=I\setminus\{0\}$.
Moreover, we put
\begin{center}
$BU\!C_0(I,X):=BU\!C(I,X)$\hspace*{1em} and \hspace*{1em} $BU\!C^1_0(I,X):=BU\!C^1(I,X)$.
\end{center}
In addition, if $I=[0,T)$ is a half open interval, then
\begin{align*}
&C_{1-\gamma}(I,X):=\{v\in{C(\dot{I},X)}:v\in{BU\!C_{1-\gamma}([0,t],X)},\hspace{.5em} t<T\},\\
&C^1_{1-\gamma}(I,X):=\{v\in{C^1(\dot{I},X)}:v,\dot{v}\in{C_{1-\gamma}(I,X)}\}.
\end{align*}
We equip these two spaces with the natural Fr\'echet topology induced by the topology of $BU\!C_{1-\gamma}([0,t],X)$ and $BU\!C_{1-\gamma}^1([0,t],X)$, respectively. 

Assume that $E_1\overset{d}{\hookrightarrow} E_0$ is a densely embedded Banach couple. Define
\begin{align}
\label{S3: ez&ef}
{\ez}(I):=BU\!C_{1-\gamma}(I,E_0), \hspace{1em} {\ef}(I):=BU\!C_{1-\gamma}(I,E_1)\cap{BU\!C_{1-\gamma}^1(I,E_0)}.
\end{align}
For $A \in\cH(E_1,E_0)$, we say $({\ez}(I),{\ef}(I))$ is a pair of {\em maximal regularity} of $A$, if
\begin{center}
$(\frac{d}{dt}+A,\gamma_{0})\in \Lis({\ef}(I),{\ez}(I)\times{E_\gamma})$, 
\end{center}
where $\gamma_0$ is the evaluation map at $0$, i.e., $\gamma_0(u)=u(0)$, and $E_\gamma:=(E_0,E_1)_{\gamma,\infty}^0$. Symbolically, we denote this property by
\begin{align*}
A \in \cM_\gamma(E_1,E_0).
\end{align*}

Now following a well-known theorem by G.~Da~Prato and P.~Grisvard \cite{DaPra79} and S.~Angenent \cite{Ange90} and the proof of \cite[Theorem~3.7]{ShaoSim13}, we have
\begin{theorem}
\label{S3: MR}
Let $\gamma\in(0,1]$, $s\in\R_+\setminus\N$ and  $\vartheta\in\R$. Suppose that $\cA$ satisfies the conditions in Proposition~\ref{S3: Prop-res-est}. Then
$$\cA\in\cM_\gamma(bc^{s+2l,\vartheta}(\M,V), bc^{s,\vartheta}(\M,V)).$$
\end{theorem}

\begin{remark}
\label{S3: MR-lower reg}
In order to prove the statement in Theorem~\ref{S3: MR}, it suffices to require  $(\M,g; \rho)$ to be a $C^{2l+[s]+1}$-singular manifold.
\end{remark}

\subsection{\bf Domains with compact boundary as singular manifolds}
Suppose that $\Omega\subset \R^m$ is a $C^k$-domain with compact boundary for $k> 2$. Then 
$\Omega$ satisfies a {\em uniform exterior and interior ball condition}, i.e., there is some $r>0$ such that for every $x\in \partial\Omega$ there are balls $\B(x_i,r)\subset \Omega$ and $\B(x_e,r)\subset \R^m\setminus\Omega$ such that 
$$\partial\Omega\cap \bar{\B}(x_i,r)=\partial\Omega\cap \bar{\B}(x_e,r)=x.$$ 
For $a\leq r$, we denote the $a$-tubular neighborhood of $\partial\Omega$ by ${\sf T}_a$.
Let 
$$d_{\partial\Omega}(x):={\sf dist}(x, \partial\Omega),\quad x\in \Omega,$$ 
i.e., the distance function to the boundary.
We define ${\sf d}: \Omega\rightarrow \R^+$ by 
\begin{equation}
\label{S3.2: rescaled dist}
{\sf d}=d_{\partial\Omega} \quad \text{if $\Omega$ is bounded},\quad \text{or} \quad 
\begin{cases}
{\sf d}=d_{\partial\Omega} \quad &\text{in } \Omega\cap {\sf T}_a,\\
{\sf d}\sim {\bf 1} &\text{in }\Omega\setminus {\sf T}_a
\end{cases}
\quad \text{otherwise}.
\end{equation}
Then we have the following proposition.
\begin{prop}
\label{S3.2: domain is Sing mnfd}
Let $\beta\geq 1$. Suppose that $\Omega\subset \R^m$ is a $C^k$-domain with compact boundary and $k>2$. Then
$(\Omega, g_m; {\sf d}^\beta)$ is a $C^{k-1}$-singular manifold.
\end{prop}
\begin{proof}
The case of $k=\infty$ is a direct consequence of \cite[Theorem~1.6]{Ama14}. When $k<\infty$, one notices that, to parameterize ${\sf T}_a$, we need to use the outward pointing unit normal of $\partial\Omega$, which is $C^{k-1}$-continuous. By a similar argument to \cite[Theorem~1.6]{Ama14}, we can then prove the asserted statement. 
\end{proof}

Given any finite dimensional Banach space $X$, by defining the singular manifold  $(\M,g;\rho)$ by $(\Omega,g_m; {\sf d}^\beta)$, we denote the weighted {\em little H\"older} spaces defined on $\Omega$ by $b^{s,\vartheta}_\beta(\Omega,X)$, i.e., $b^{s,\vartheta}_\beta(\Omega,X)=bc^{s,\vartheta}(\M,X)$. 

In view of Remark~\ref{S3: MR-lower reg}, we have the following continuous maximal regularity theorem for elliptic operators with higher order degeneracy on domains.
\begin{theorem}
\label{S3: MR-domain}
Let $\gamma\in(0,1]$, $s\in\R_+\setminus\N$, $\vartheta\in\R$, $\beta\geq 1$ and $k=2l+[s]+2$. Suppose that $\Omega\subset\R^m$ is a $C^k$-domain and the differential operator
$$\cA:=\sum\limits_{|\alpha|\leq 2l} a_\alpha \partial^\alpha$$ 
satisfies  
\begin{itemize}
\item[(a)] for any $\xi\in\mathbb{S}^{m-1}$ 
$$S:=\Sigma_{\pi/2} \subset \rho(-\hat{\sigma}\cA^\pi(x,\xi)), $$
and for some $\Ce>0$
$$ ({\sf d}^{2l\beta} (x) + |\mu|) \| (\mu +\hat{\sigma}\cA^\pi(x,\xi))^{-1}  \|_{\L(X)} \leq \Ce , \quad \mu\in S; $$
\item[(b)] $a_\alpha \in bc^{s, -|\alpha|}_\beta(\Omega,\L(X))$.
\end{itemize}
Then
$$\cA\in\cM_\gamma(bc^{s+2l,\vartheta}_\beta(\Omega,X), bc^{s,\vartheta}_\beta(\Omega,X)).$$
\end{theorem}
The above theorem generalizes the results of \cite{ForMetPall11, Vesp89} to unbounded domains and elliptic operators with order higher than two.
\begin{rmk}
\begin{itemize}
\item[] {\phantom{some stuff}}
\item[(a)] Condition~(a) in Theorem~\ref{S3: MR-domain} can be replaced by the following condition.  For any $\xi\in\mathbb{S}^{m-1}$ and $\eta\in X$, 
$$\langle \hat{\sigma}\cA^\pi(x,\xi)\eta , \eta \rangle_X \sim {\sf d}^{2l\beta}|\eta|_X^2 .$$
Here $\langle \cdot, \cdot \rangle$ is the inner product in $X$.
\item[(b)] In Theorem~\ref{S3: MR-domain}, taking $X$ to be any infinite dimensional Banach space is also admissible.
\end{itemize}

\end{rmk}

\section{\bf Applications}

\subsection{\bf The porous medium equation}
We consider the porous medium equation on  a singular manifold $(\M,g;\rho)$, which reads as follows.
\begin{equation}
\label{S4: porous-eq}
\left\{\begin{aligned}
\partial_t u -\Delta u^n &=f ;\\
u(0)&=u_0 &&
\end{aligned}\right.
\end{equation}
for $n>1$. Let 
$$P(u):= -n u^{n-1}\Delta, \quad Q(u):=n(n-1) |\gd u|_g^2 u^{n-2}.$$
Here $\Delta:=\Delta_g$ with $\Delta_g$ standing for the Laplacian-Beltrami operator with respect to $g$. 
A direct computation shows that equation~\eqref{S4: porous-eq} is equivalent to
\begin{equation*}
\left\{\begin{aligned}
\partial_t u +P(u) u&=Q(u)+ f ;\\
u(0)&=u_0 . &&
\end{aligned}\right.
\end{equation*}
Given any $0<s<1$, put $\vartheta=-2/(n-1)$.
In the current context, $V=\R$, thus we abbreviate the notation $bc^{s^\prime,\vartheta}(M,V)$ to $bc^{s^\prime,\vartheta}(\M)$ for any $s^\prime\geq 0$.
Let 
$$E_0:=bc^{s,\vartheta}(\M),\quad E_1:=bc^{2+s,\vartheta}(\M),\quad E_{1/2}:=(E_0,E_1)^{0}_{1/2,\infty}.$$
Then by Proposition~\ref{S2: interpolation}, $E_{1/2}\doteq bc^{1+s,\vartheta}(\M)$. Let 
$$U^{1+s}_\vartheta :=\{u\in E_{1/2}:\inf\rho^{\vartheta} u>0\},$$ 
which is open in $E_{1/2}$.

For any $\beta\in\R$, define ${\sf P}_{\beta}:U^{1+s}_\vartheta\rightarrow L_{1,loc}(\M):$ $u\mapsto u^\beta$. 
One readily checks that \cite[Proposition~6.3]{ShaoSim13} still holds true for singular manifolds. 
Hence by \cite[Proposition~6.3]{ShaoSim13} and Proposition~\ref{S2: change of wgt}, we obtain 
\begin{align}
\label{S4: power of u}
[u\mapsto u^{\beta}]=[u\mapsto \rho^{-\beta\vartheta}{\sf P}_\beta(\rho^{\vartheta}u)]\in C^\omega(U^{1+s}_\vartheta,bc^{1+s,\beta\vartheta}(\M)).
\end{align}
In view of  (P2), we infer that $\Rc g^* \in l^2_\infty(\boldsymbol{BC}^k(E^2_0))$ for any $k\in\Nz$. Then  Proposition~\ref{S2: retraction&coretraction} yields
\begin{align}
\label{S4: reg of g^*}
g^*\in BC^{\infty,2}(\M,V^2_0).
\end{align}
One may check via Proposition~\ref{S2: pointwise multiplication}, \eqref{S4: power of u} and \eqref{S4: reg of g^*} that
$$u^{n-1} g^* \in bc^{1+s}(\M,V^2_0), \quad u\in U^{1+s}_\vartheta.$$
On account of the expression $\Delta_{g}v=\ev(g^*,\nabla^2 v)$, 
it is then a direct consequence of Proposition~\ref{S3: diff-op} and \cite[Proposition~1]{Brow62} that 
\begin{align}
\label{S4.1: P-reg}
P \in C^{\omega}(U^{1+s}_\vartheta,\L(E_1,E_0)).
\end{align}
In the above, $\nabla:=\nabla_g$, where $\nabla_g$ is Levi-Civita connection of $g$.
Given any $\vartheta^\prime\in\R$, by Proposition~\ref{S2: nabla} and \cite[Proposition~2.5]{ShaoPre}, one obtains 
\begin{equation}
\label{S4.1: grad-BC^k}
\gd \in \L( BC^{k+1,\vartheta^\prime}(\M,V^\sigma_\tau), BC^{k,\vartheta^\prime+2}(\M,V^{\sigma+1}_\tau)).
\end{equation}
A density argument as in the proof for Proposition~\ref{S2: change of wgt} yields
$$
\gd \in \L( bc^{k+1,\vartheta^\prime}(\M,V^\sigma_\tau), bc^{k,\vartheta^\prime+2}(\M,V^{\sigma+1}_\tau)).
$$
Interpolation theory and definition~\eqref{S2.2: def-BC^s} implies that \eqref{S4.1: grad-BC^k} also holds for {\em H\"older} spaces of non-integer order. Applying the density argument as in the proof for Proposition~\ref{S2: change of wgt} once more, we establish the assertion for weighted {\em little H\"older} spaces of non-integer order, that is, for any $s^\prime\geq 0$ 
\begin{align}
\label{S4: grad}
\gd \in \L( bc^{s^\prime+1,\vartheta^\prime}(\M,V^\sigma_\tau),bc^{s^\prime,\vartheta^\prime+2}(\M,V^{\sigma+1}_\tau)).
\end{align}
We have the expression $|\gd u|_g^2=\ev(\nabla u, \gd u)$. Since complete contraction is a bundle multiplication,  
we infer from Propositions~\ref{S2: pointwise multiplication}, \ref{S2: nabla} and \eqref{S4: grad} that
\begin{align}
\label{S4: u-|u|}
[u\mapsto |\gd u|^2_g]\in C^\omega(U^{1+s}_\vartheta, bc^{s,2\vartheta+2}(\M)).
\end{align} 
Proposition~\ref{S2: pointwise multiplication}, \eqref{S4: power of u} and \eqref{S4: u-|u|} immediately imply
\begin{align}
\label{S4.1: Q-reg}
Q\in C^\omega(U^{1+s}_\vartheta, E_0).
\end{align}
Given any $u\in U^{1+s}_\vartheta$, one verifies that the principal symbol of $P(u)$ fulfils
$$-n \ev(u^{n-1} g^*, (-i\xi)^{\otimes 2})=n\rho^2 (\rho^\vartheta u)^{n-1} |\xi|^2_{g^*}\geq n (\inf \rho^\vartheta u)^{n-1} \rho^2|\xi|^2_{g^*},    $$
for any cotangent field $\xi$. Hence for any $u\in U^{1+s}_\vartheta$, $P(u)$ is \emph{normally $\rho$-elliptic}. It follows from Theorem~\ref{S3: MR} that
\begin{align}
\label{S4.1: MR-PME}
P(u)\in \cM_{\gamma}(E_1,E_0),\hspace{1em}u\in U^{1+s}_\vartheta.
\end{align}

\begin{theorem}
\label{S4.1: Thm-porous}
Suppose that $u_0\in U^{1+s}_\vartheta:=\{u\in bc^{1+s,\vartheta}(\M):\inf\rho^{\vartheta} u>0 \}$ with $0<s<1$,  $\vartheta=-2/(n-1)$, and $f\in bc^{s,\vartheta}(\M)$. Then equation \eqref{S4: porous-eq} has a unique local positive solution 
\begin{align*}
\hat{u}\in C^1_{1/2}(J(u_0),bc^{s,\vartheta}(\M))\cap C_{1/2}(J(u_0),bc^{2+s,\vartheta}(\M)) \cap C(J(u_0),U^{1+s}_\vartheta)
\end{align*}
existing on $J(u_0):=[0,T(u_0))$ for some $T(u_0)>0$. Moreover, 
$$\hat{u} \in  C^{\infty}(\dot{J}(u_0)\times \M) .$$
Here $\dot{J}:= J \setminus\{0\}$. 
\end{theorem}
\begin{proof}
In virtue of \eqref{S4.1: P-reg}, \eqref{S4.1: Q-reg} and \eqref{S4.1: MR-PME}, \cite[Theorem~4.1]{CleSim01} immediately establishes the local existence and uniqueness part. The short term positivity of the solution follows straightaway from the continuity of the solution. To argue for the asserted regularity property of the solution $\hat{u}$, we look at $v:=\rho^\vartheta \hat{u}$. By multiplying both sides of equation~\ref{S4: porous-eq} with $\rho^\vartheta$, we have
\begin{equation*}
\left\{\begin{aligned}
\partial_t v -\rho^\vartheta \Delta \rho^{2-\vartheta} v^n &=\rho^\vartheta f ;\\
v(0)&=\rho^\vartheta u_0. &&
\end{aligned}\right.
\end{equation*}
One checks
\begin{align*}
\rho^\vartheta \Delta \rho^{2-\vartheta} v^n=& n\rho^2 v^{n-1} \Delta v + n(n-1)\rho^2 |\gd v|^2_g v^{n-2} \\
&+2n(2-\vartheta) \rho^2(\gd \log \rho |\gd v)_g v^{n-1}\\
&+(2-\vartheta)[\rho \Delta\rho+ (1-\vartheta)|\gd \rho|^2_g ] v^n.
\end{align*}
Let $\hat{g}=g/\rho^2$. 
Recall that $(\M,\hat{g})$ is a uniformly regular Riemannian manifold. 
Put $U^{1+s}:=\{v\in bc^{1+s}(\M): \inf v>0 \}$. 
By \cite[formula~(5.15)]{Ama13b}, 
$$\rho^2|\gd v|^2_g=|\gd_{\hat{g}} v|^2_{\hat{g}}.$$
We have
$$(\gd \log \rho |\gd v)_g = (\gd \log \rho |\gd_{\hat{g}}v)_{\hat{g}}. $$
It follows from \cite[formula~(5.8)]{Ama13b} that $\rho^2\gd \log\rho\in BC^{1,0}(\M, T\M)$. By Proposition~\ref{S2.3: unweighted spaces},
$$\rho^2\gd\log\rho \in BC^1(\M, \widehat{T\M}).$$
\cite[formula~(5.6)]{ShaoSim13} implies $v^{n-1}\in bc^{1+s}(\M)$ for all $v\in U^{1+s}$. 
By Proposition~\ref{S2: retraction&coretraction} and (S3), we can show that 
$$\rho \Delta\rho+ (1-\vartheta)|\gd \rho|^2_g \in BC^1(\M).$$
Put
$$P(v):= -n\rho^2 v^{n-1} \Delta ,$$ 
and
\begin{align*}
Q(v):=&\rho^\vartheta \Delta \rho^{2-\vartheta} v^n +P(v)v\\
=&n(n-1)\rho^2 |\gd v|^2_g v^{n-2} 
+2n(2-\vartheta) \rho^2(\gd \log \rho |\gd v)_g v^{n-1}\\
&+(2-\vartheta)[\rho \Delta\rho+ (1-\vartheta)|\gd \rho|^2_g ] v^n.
\end{align*}

Then by the above discussion, we infer that
$$P\in C^\omega(U^{1+s}, \L(bc^{2+s}(\M),bc^s(\M))), \quad Q\in C^\omega(U^{1+s}, bc^s(\M)). $$
For each $v\in U^{1+s}$, we can check that $P(v)$ is {\em normally elliptic} in the sense of \cite[Section~3]{ShaoSim13}. Applying the parameter-dependent  diffeomorphism technique in \cite{Shao13}, we can establish 
$$v \in   C^{\infty}(\dot{J}(u_0)\times \M) ,$$
which in turn implies
$$\hat{u} \in  C^{\infty}(\dot{J}(u_0)\times \M) .$$
\end{proof}

\begin{remark}
It is clear Theorem~\ref{S4.1: Thm-porous} still holds true for the fast diffusion case of the porous medium equation (the plasma equation).
\end{remark}

Before concluding this subsection, we comment on the Cauchy problem for the porous medium equation and its waiting-time phenomenon. Since our conclusion for the porous medium equation, to some extend, can be viewed as a simpler version of the corresponding theory of the thin film equation in Section~4.4, we will only state our results without providing proofs. More details can be found in Section~4.4.
\begin{remark}
Suppose that $\supp(u_0)=:\Omega\subset\R^m$ is a $C^4$-domain with compact boundary, and $u_0\in U^{1+s}_\vartheta:=\{u\in bc^{1+s,\vartheta}_1(\Omega):\inf{\sf d}^{\vartheta} u>0 \}$ with $0<s<1$,  $\vartheta=-2/(n-1)$. We learn from Proposition~\ref{S3.2: domain is Sing mnfd} that $(\Omega,g_m;{\sf d})$ is a $C^3$-singular manifold, where ${\sf d}$ is defined in \eqref{S3.2: rescaled dist}. Then by Theorems~\ref{S3: MR-domain} and~\ref{S4.1: Thm-porous}, for every $f\in bc^{s,\vartheta}_1(\Omega)$, the equation
\begin{equation*}
\label{S4.1: thin-film-eq-domain}
\left\{\begin{aligned}
\partial_t u +\Delta u^n &=f  &&\text{on}&&\Omega_T; \\
u(0)&=u_0  &&\text{on }&&\Omega ,&&
\end{aligned}\right.
\end{equation*}
with $\Omega_T:=\Omega\times (0,T)$, has a unique solution 
\begin{align}
\label{S4.1: solu of PME-domain}
\hat{u}\in C^1_{1/2}(J(u_0),bc^{s,\vartheta}_1(\Omega))\cap C_{1/2}(J(u_0),bc^{2+s,\vartheta}_1(\Omega)) \cap C(J(u_0),U^{1+s}_\vartheta).
\end{align}
Furthermore, by identifying $\hat{u},f,u_0\equiv 0$ in $\R^m\setminus\Omega$, $\hat{u}$ is indeed a strong $L_1$-solution of the Cauchy problem
\begin{equation*}
\label{S4.1: thin-film-eq-Cauchy}
\left\{\begin{aligned}
\partial_t u +\Delta u^n &=f  &&\text{on}&&\R^m_T; \\
u(0)&=u_0  &&\text{on }&&\R^m &&
\end{aligned}\right.
\end{equation*}
in the sense of \cite[Definition~9.1]{Vaz07}, except that the interval of existence $[0,\infty)$ in \cite[Definition~9.1]{Vaz07} is replaced by $J(u_0)$. This solution is unique by \cite[Theorem~9.2]{Vaz07}. Another observation from \eqref{S4.1: solu of PME-domain} is that $\hat{u}$ enjoys the so-called waiting-time property, that is, 
\begin{equation*}
\supp[\hat{u}(t,\cdot)] =\supp[\hat{u}(0,\cdot)], \quad t\in (0,T(u_0)). 
\end{equation*}
 
\end{remark}

\subsection{\bf The Yamabe flow}
Suppose that $(\M,g_0;\rho)$ is a singular manifold without boundary  of dimension $m$ for $m\geq 3$. The Yamabe flow reads as
\begin{equation}
\label{S4: Yamabe flow eq1}
\left\{\begin{aligned}
\partial_t g&=-R_g g ;\\
g(0)&=g^0, &&
\end{aligned}\right.
\end{equation}
where $R_g$ is the scalar curvature with respect to the metric $g$. $g^0$ is in the conformal class of the  background metric $g_0$ of $\M$.

We seek solutions to the Yamabe flow \eqref{S4: Yamabe flow eq1} in the conformal class of the metric $g_0$. 
Let $c(m):=\frac{m-2}{4(m-1)}$, and define the conformal Laplacian operator $L_g$ with respect to the metric $g$ as:
$$
L_gu:=\Delta_g u -c(m)R_g u.
$$
Let $g=u^{\frac{4}{m-2}}g_0$ for some $u>0$. It is well known that by rescaling the time variable equation \eqref{S4: Yamabe flow eq1} is equivalent to 
\begin{equation*}
\left\{\begin{aligned}
\partial_t u^{\frac{m+2}{m-2}}&=\frac{m+2}{m-2}L_0 u; \\
u(0)&=u_0,&&
\end{aligned}\right.
\end{equation*}
where $L_0:=L_{g_0}$ and $u_0$ is a positive function. See \cite[formula~(7)]{MaCheZhu12}. It is equivalent to solving the following equation:
\begin{equation}
\label{S4: Yamabe flow eq2}
\left\{\begin{aligned}
\partial_t u &=u^{-\frac{4}{m-2}}L_0 u; \\
u(0)&=u_0. &&
\end{aligned}\right.
\end{equation}
A well-known formula of scalar curvature in local coordinates yields
\begin{align*}
R_g=\frac{1}{2}g^{ki}g^{lj}(g_{jk,li}+g_{il,kj}-g_{jl,ki}-g_{ik,lj}).
\end{align*}
(P2) implies that 
$$\Rc R_{g_0} \in l_\infty^2(\boldsymbol{BC}^k(\R)), $$
for any $k\in \Nz$. 
By Proposition~\ref{S2: retraction&coretraction}, we infer that
\begin{align}
\label{S4.2: reg of R_0}
R_{g_0}\in BC^{\infty,2}(\M).
\end{align}
Put 
$$P(u)h:=-u^{-\frac{4}{m-2}}\Delta_{g_0}h ,\quad Q(u):=-c(m) u^{\frac{m-6}{m-2}}R_{g_0} .$$

Given any $0<s<1$, we choose $0<\alpha<s$, $\gamma=(s-\alpha)/2$. 
Let $\vartheta=(m-2)/2$ and
$$E_0:=bc^{\alpha,\vartheta}(\M),\quad E_1:=bc^{2+\alpha,\vartheta}(\M),\quad E_{\gamma}:=(E_0,E_1)^{0}_{\gamma,\infty}.$$
Then by Proposition~\ref{S2: interpolation}, $E_\gamma \doteq bc^{s,\vartheta}(\M)$. Put 
$$U^s_\vartheta=\{u\in E_{\gamma}:\inf\rho^{\vartheta} u>0\}.$$
In view of \eqref{S4.2: reg of R_0}, it follows from an analogous discussion as in \eqref{S4.1: P-reg} and \eqref{S4.1: Q-reg} that
\begin{align}
\label{S4: P-Q-reg}
P \in C^{\omega}(U^s_\vartheta,\L(E_1,E_0)),\quad Q\in C^{\omega}(U^s_\vartheta,E_0).
\end{align}
A similar computation as in \eqref{S4.1: MR-PME} yields
\begin{equation}
\label{S4.2 Yamabe-MR}
P(u)\in \mathcal{M}_{\gamma}(E_1,E_0),\quad u\in U^s_\vartheta.
\end{equation}
\begin{theorem}
\label{S4.1: Thm-Yamabe}
Suppose that $u_0\in U^s_\vartheta:=\{u\in bc^{s,\vartheta}(\M):\inf\rho^{\vartheta} u>0 \}$ with $0<s<1$, and $\vartheta=(m-2)/2$. Then for every fixed $\alpha\in (0,s)$, equation \eqref{S4: Yamabe flow eq2} has a unique local positive solution 
\begin{align*}
\hat{u}\in C^1_{1-\gamma}(J(u_0),bc^{\alpha,\vartheta}(\M))\cap C_{1-\gamma}(J(u_0),bc^{2+\alpha,\vartheta}(\M)) \cap C(J(u_0),U^s_\vartheta)
\end{align*}
existing on $J(u_0):=[0,T(u_0))$ for some $T(u_0)>0$ with $\gamma=(s-\alpha)/2$. Moreover, 
$$\hat{g}\in C^{\infty}(\dot{J}(u_0)\times \M, V^0_2).$$
In particular, if the metric $g_0/\rho^2$ is real analytic, then
$$\hat{g}\in C^\omega(\dot{J}(u_0)\times \M, V^0_2).$$
\end{theorem}
\begin{proof}
Local existence and uniqueness is a direct consequence of \eqref{S4: P-Q-reg}, \eqref{S4.2 Yamabe-MR}, and \cite[Theorem~4.1]{CleSim01}. The  regularity part follows in a similar way to the proof of Theorem~\ref{S4.1: Thm-porous}.
\end{proof}
\begin{remark}
The scalar of the initial metric $g^0= u_0^{\frac{4}{m-2}} g_0$ is related that of the background singular metric $g_0$ in the following manner:
\begin{align}
\label{S2: scalar cur-conformal change}
R_{g_0} = -\frac{4(m-1)}{m-2} u_0^{-\frac{m+2}{m-2}} L_{\tilde{g}} u.
\end{align}
We may take $\rho={\bf 1}_\M$ for computational brevity, i.e., $(\M,g_0)$ to be {\em uniformly regular}. Then there is some $C>1$ such that 
$$1/C\leq \|u_0^{-\frac{m+2}{m-2}}\|_\infty \leq C,\quad \|R_{g_0}\|_\infty \leq C.$$ 
But at the same time, there are ample examples of $u_0 \in U^s_\vartheta$ with unbounded derivatives.
In view of formula~\eqref{S2: scalar cur-conformal change}, 
it is not hard to create $g^0$ with unbounded scalar curvature. 
Therefore, the Yamabe flow can admit a unique smooth solution while starting at a metric with unbounded curvature, and these solutions evolve into one with bounded curvature instantaneously.

The initial metric $g^0 = u_0^{\frac{4}{m-2}} g_0$ in the above theorem can have unbounded scalar curvature. To make this already long paper not any longer, we will give more details on this observation elsewhere.
\end{remark}

\subsection{\bf The evolutionary $p$-Laplacian equation}
In this subsection, we investigate the well-posedness of the following evolutionary $p$-Laplacian equation on a singular manifold $(\M,g;\rho)$.
\begin{equation}
\label{S4: p-Lap-eq}
\left\{\begin{aligned}
\partial_t u -\div (|\gd u|_g^{p-2}\gd u)&=f;  \\
u(0)&=u_0 . &&
\end{aligned}\right.
\end{equation}
Here $p>1$ with $p\neq 2$, and $\gd=\gd_g$, $\div=\div_g$.
One computes
\begin{align*}
\div (|\gd u|_g^{p-2}\nabla u)&=|\gd u|_g^{p-2} \Delta u + (p-2)|\gd u|_g^{p-4} \ev((\gd u)^{\otimes 2}, \nabla^2 u)\\
&=|\gd u|_g^{p-4} \ev(|\gd u|^2 g^* +(p-2)(\gd u)^{\otimes 2}, \nabla^2 u).
\end{align*}
Let
$$\vec{a}(u):= -|\gd u|_g^{p-4} (|\gd u|^2 g^* +(p-2)(\gd u)^{\otimes 2}).$$
For any $0<s<1$, we put $\vartheta=p/(2-p)$ and
$$E_0:=bc^{s,\vartheta}(\M),\quad E_1:=bc^{2+s,\vartheta}(\M),\quad E_{1/2}:=(E_0,E_1)^{0}_{1/2,\infty}.$$
Proposition~\ref{S2: interpolation} implies $E_{1/2}\doteq bc^{1+s,\vartheta}(\M)$. Let 
$$U^{1+s}_\vartheta:=\{u\in E_{1/2}:\inf\rho^{\vartheta+1}|\gd u|_g>0\}.$$ 
This is an open subset of $E_{1/2}$. 

We infer from \eqref{S4: power of u} and \eqref{S4: u-|u|} that
$$[u\mapsto |\gd u|_g^{p-2}]\in C^\omega(U^{1+s}_\vartheta, bc^{s,-2}(\M)), $$
and from \cite[Example~13.4(b)]{AmaAr}, Proposition~\ref{S2: pointwise multiplication} and \cite[Proposition~2.5]{ShaoPre} that
$$[u\mapsto |\gd u|_g^{p-4}(\gd u)^{\otimes 2}]\in C^\omega(U^{1+s}_\vartheta, bc^{s,0}(\M,V^2_0) ). $$
In virtue of \eqref{S4: reg of g^*} and Proposition~\ref{S2: pointwise multiplication}, we have
\begin{align}
\label{S4.3: vec a-reg}
[u\mapsto \vec{a}(u)]\in C^\omega(U^{1+s}_\vartheta,bc^{s,0}(\M,V^2_0)).
\end{align}
The principal symbol can be computed as in Section~4.1.
\begin{align*}
& \ev(\vec{a}(u), (-i\xi)^{\otimes 2})(\p)\\
=&  |\gd u(\p)|_{g(\p)}^{p-2} |\xi(\p)|_{g^*(\p)}^2 +(p-2)|\gd u(\p)|_{g(\p)}^{p-4}[\ev(\gd u, \xi)(\p)]^2\\
=&  |\gd u(\p)|_{g(\p)}^{p-2} |\xi(\p)|_{g^*(\p)}^2 +(p-2)|\gd u(\p)|_{g(\p)}^{p-4}(\nabla u (\p)| \xi(\p))^2_{g^*(\p)} .
\end{align*}
For $p>2$, one checks  for any $\xi\in \Gamma(\M, T^*\M)$
\begin{align*}
 \ev(\vec{a}(u), (-i\xi)^{\otimes 2})(\p)
&\geq  |\gd u(\p)|_{g(\p)}^{p-2} |\xi(\p)|^2_{g^*(\p)}\\
& \geq (\inf \rho^{\vartheta+1}|\gd u|_g)^{p-2} \rho^2(\p) |\xi(\p)|^2_{g^*(\p)},
\end{align*}
and for $1<p<2$
\begin{align*}
&\quad \ev(\vec{a}(u), (-i\xi)^{\otimes 2})(\p)\\
&\geq |\gd u(\p)|_{g(\p)}^{p-2} |\xi(\p)|_{g^*(\p)}^2 +(p-2)|\gd u(\p)|_{g(\p)}^{p-2} |\xi(\p)|_{g^*(\p)}^2\\
&= (p-1)|\gd u(\p)|_{g(\p)}^{p-2} |\xi(\p)|^2_{g^*(\p)}\\
& \geq (p-1) (\sup \rho^{\vartheta+1}|\gd u|_g)^{p-2} \rho^2(\p) |\xi(\p)|^2_{g^*(\p)},
\end{align*}
holds for all $u\in U^{1+s}_\vartheta.$ In the second step, we have used the Cauchy-Schwarz inequality. 
Therefore, $\ev(\vec{a}(u), \nabla^2 \cdot)$ is {\em normally $\rho$-elliptic} for every $u\in U^{1+s}_\vartheta$.

\begin{theorem}
\label{S4.3: Thm-p-Lap}
Suppose that $u_0\in U^{1+s}_\vartheta:=\{u\in bc^{1+s,\vartheta}(\M):\inf\rho^{\vartheta+1}|\gd u|_g>0\}$ with $0<s<1$, $\vartheta=p/(2-p)$, and $f\in bc^{s,\vartheta}(\M)$. Then equation~\eqref{S4: p-Lap-eq} has a unique local solution 
\begin{align*}
\hat{u}\in C^1_{1/2}(J(u_0),bc^{s,\vartheta}(\M))\cap C_{1/2}(J(u_0),bc^{2+s,\vartheta}(\M)) \cap C(J(u_0),U^{1+s}_\vartheta)
\end{align*}
existing on $J(u_0):=[0,T(u_0))$ for some $T(u_0)>0$. Moreover,
$$\hat{u}\in C^{\infty}(\dot{J}(u_0)\times \M).$$
\end{theorem}
\begin{proof}
The assertion follows in a similar way to the proof of Theorem~\ref{S4.1: Thm-porous}.
\end{proof}

\begin{remark}
\label{S4.3: boundary-blowup}
Suppose that we take $(\M,g;\rho)$ to be $(\Omega,g_m; {\sf d}^\beta)$ with $\beta \geq 1$ for some $C^4$-domain $\Omega\subset\R^m$ with compact boundary.  Let $\vartheta=p /(2-p)$. 
If we define 
$$U^{1+s}_{\vartheta}:=\{u\in bc^{1+s,\vartheta}_\beta(\Omega):\inf{\sf d}^{(\vartheta+1)\beta}|D u|>0,\, \inf{\sf d}^{\beta\vartheta} u>0\},$$ 
then Theorem~\ref{S4.3: Thm-p-Lap} still holds true, where $D$ is the gradient with respect to the metric $g_m$. Then there exists a positive continuous function $c(t)$ in $J:=J(u_0)$ such that
$$ {\sf d}^{\beta\vartheta}(x) \hat{u}(t,x) \geq c(t),\quad t\in J.  $$
In particular, the  above inequality shows that, for $1<p<2$, as $x\to \partial\Omega$
$$\hat{u}(t,x)\geq c(t) {\sf d}^{\frac{p\beta}{p-2}}(x)\to \infty, \quad t\in J. $$
This validates the assertion about equation~\eqref{S1: boundary blowup} in Section~1.
\end{remark}

\subsection{\bf The thin film equation on domains}
Suppose that $\Omega\subset \R^m$ is a $C^6$-domain with compact boundary. Then by the discussion in Section~3.2, $(\Omega,g_m ; {\sf d}^\beta)$ with $\beta\geq 1$ is a singular manifold, where ${\sf d}$ is defined in \eqref{S3.2: rescaled dist}. We consider the following thin film equation with $n>0$  and  degenerate boundary condition. Physically, the power exponent is determined by the flow condition at the liquid-solid interface, and is usually constrained to $n\in (0,3]$. Since the other choices of $n$ make no difference in our theory, $n\in [3,\infty)$ is also included herein.
\begin{equation}
\label{S4: thin-film-eq}
\left\{\begin{aligned}
\partial_t u +\div (u^n D \Delta u +\alpha_1 u^{n-1}\Delta u D u +\alpha_2 u^{n-2} |D u|^2 D u)&=f  &&\text{on}&&\Omega_T; \\
u(0)&=u_0  &&\text{on }&&\Omega .&&
\end{aligned}\right.
\end{equation}
Here $\alpha_1,\alpha_2$ are two constants, and $D$ denotes the gradient in $\R^m$.
An easy computation shows that
\begin{align*}
&\quad \div (u^n D \Delta u +\alpha_1 u^{n-1}\Delta u D u +\alpha_2 u^{n-2} |D u|^2 D u)\\
&= u^n \Delta^2 u + (n +\alpha_1) u^{n-1} (D u| D \Delta u)_{g_m} +\alpha_1 u^{n-1} (\Delta u)^2\\
&\quad   + [\alpha_1(n-1) +\alpha_2] u^{n-2}|D u|^2 \Delta u +\alpha_2(n-2)u^{n-3} |D u|^4\\
&\quad + 2\alpha_2 u^{n-2}(\nabla^2 u Du | Du)_{g_m}.
\end{align*}
For any $0<s <1$, take $\vartheta=-4/n$ 
$$E_0:=bc^{s,\vartheta}_\beta(\Omega),\quad E_1:=bc^{4+s,\vartheta}_\beta(\Omega),\quad E_{1/2}=(E_0,E_1)^0_{1/2,\infty}.$$
Then $ E_{1/2}\doteq bc^{2+s,\vartheta}_\beta(\Omega)$. 
Let $U^{2+s}_\vartheta:=\{u\in E_{1/2}: \inf {\sf d}^{\beta\vartheta} u>0\}$.  For any $u\in U^{2+s}_\vartheta$ and $v\in E_1$, we define
\begin{align*}
P(u)v:=&u^n \Delta^2 v + (n +\alpha_1) u^{n-1} (D u| D \Delta v)_{g_m} +\alpha_1 u^{n-1} \Delta u \Delta v\\
& + [\alpha_1(n-1) +\alpha_2] u^{n-2}|D u|^2 \Delta v +\alpha_2(n-2)u^{n-4} |D u|^4 v\\
&+ 2\alpha_2 u^{n-2}(\nabla^2 v Du | Du)_{g_m}.
\end{align*}
It follows from a similar argument as in Section~4.1 that
$$P\in C^\omega(U^{2+s}_\vartheta, \L(E_1,E_0)) $$
and for every $u\in U^{2+s}_\vartheta$, the principal symbol of $P(u)$ can be computed as
\begin{align*}
\hat{\sigma}P(u)(x,\xi)&=u^n(x) (g_m((-i\xi),(-i\xi)))^2\\
& = {\sf d}^{4\beta}(x) ({\sf d}^\vartheta u)^n(x)|\xi|^4 
 \geq (\inf{\sf d}^\vartheta u )^n {\sf d}^{4\beta}(x)|\xi|^4.
\end{align*}
Thus $P(u)$ is {\em normally $\rho$-elliptic}.
\begin{theorem}
\label{S4.4: Thm-thin-film}
Given any $\beta\geq 1$, 
suppose that $u_0\in U^{2+s}_\vartheta:=\{u\in bc^{2+s,\vartheta}_\beta(\Omega):\inf{\sf d}^{\beta\vartheta} u>0 \}$ with $0<s<1$, $\vartheta=-4/n$. Then for every $f\in bc^{s,\vartheta}_\beta(\Omega)$, equation \eqref{S4: thin-film-eq} has a unique local  solution 
\begin{align*}
\hat{u}\in C^1_{1/2}(J(u_0),bc^{s,\vartheta}_\beta(\Omega))\cap C_{1/2}(J(u_0),bc^{4+s,\vartheta}_\beta(\Omega)) \cap C(J(u_0),U^{2+s}_\vartheta)
\end{align*}
existing on $J(u_0):=[0,T(u_0))$ for some $T(u_0)>0$. Moreover,
$$\hat{u}\in C^{\infty}(\dot{J}(u_0)\times \Omega).$$
\end{theorem}
\begin{proof}
The proof is essentially the same as that for Theorem~\ref{S4.1: Thm-porous} except that we use Theorem~\ref{S3: MR-domain} instead of Theorem~\ref{S3: MR}.
\end{proof}
In the case $\alpha_1=0$, we can admit lower regularity for the initial data.
\begin{cor}
\label{S4.4: cor-thin-film}
Given any $\beta\geq 1$, 
suppose that $u_0\in U^{1+s}_\vartheta=\{u\in bc^{1+s,\vartheta}_\beta(\Omega):\inf{\sf d}^{\beta\vartheta} u>0 \}$ with $0<s<1$, $\vartheta=-4/n$. Then for every $f\in bc^{s,\vartheta}_\beta(\Omega)$, equation \eqref{S4: thin-film-eq} has a unique local solution 
\begin{align*}
\hat{u}\in C^1_{3/4}(J(u_0),bc^{s,\vartheta}_\beta(\Omega))\cap C_{3/4}(J(u_0),bc^{4+s,\vartheta}_\beta(\Omega)) \cap C(J(u_0),U^{1+s}_\vartheta)
\end{align*}
existing on $J(u_0):=[0,T(u_0))$. Moreover,
$$\hat{u}\in C^{\infty}(\dot{J}(u_0)\times \Omega).$$
\end{cor}

In some literature, a more general form of the thin film equation is considered with $u^n$ replaced by $\Psi(u)=u^n +\delta u^3$ with $\delta\geq 0$ and $n\in (0,3]$. The term $\delta u^3$ is sometimes omitted because it is relatively small compared to $u^n$ for $n<3$ near the free boundary $\supp[u(t,\cdot)]$.
\begin{equation}
\label{S4: thin-film-eq-general}
\left\{\begin{aligned}
\partial_t u +\div (\Psi(u) D \Delta u +\alpha_1 u^{n-1}\Delta u D u +\alpha_2 u^{n-2} |D u|^2 D u)&=f  &&\text{on}&&\Omega_T; \\
u(0)&=u_0  &&\text{on }&&\Omega. &&
\end{aligned}\right.
\end{equation}
For any $u\in U^{2+s}_\vartheta$, it is easy to check that $u^3 \in bc^{2+s,3\vartheta}_\beta(\Omega) \hookrightarrow  bc^{2+s,n\vartheta}_\beta(\Omega)$. Now the computations shown above for equation~\eqref{S4: thin-film-eq} still hold for the new system undoubtedly.
\begin{cor}
Suppose that the conditions in Theorem~\ref{S4.4: Thm-thin-film} are satisfied. Then equation \eqref{S4: thin-film-eq-general} has a unique local solution 
\begin{align*}
\hat{u}\in C^1_{1/2}(J(u_0),bc^{s,\vartheta}_\beta(\Omega))\cap C_{1/2}(J(u_0),bc^{4+s,\vartheta}_\beta(\Omega)) \cap C(J(u_0),U^{2+s}_\vartheta)
\end{align*}
existing on $J(u_0):=[0,T(u_0))$ for some $T(u_0)>0$. Moreover,
$$\hat{u}\in C^{\infty}(\dot{J}(u_0)\times \Omega).$$
\end{cor}

\begin{remark}
We may observe that the solution $\hat{u}$ obtained in Theorem~\ref{S4.4: Thm-thin-film} is actually a solution to the following initial value problem with conditions on the free boundary $\partial[\supp(u)]$. Indeed, assume that $\supp(u_0)=\bar{\Omega}$ and $\Omega$ is a $C^6$-domain with compact boundary. Let $\Omega(t):=\supp[u(t,\cdot)]$. If the initial data $u_0$ satisfies the conditions in Theorem~\ref{S4.4: Thm-thin-film}, then 
\begin{equation}
\label{S4.4: thin-film-IBVP}
\left\{\begin{aligned}
\partial_t u +\div (u^n D \Delta u +\alpha_1 u^{n-1}\Delta u D u +\alpha_2 u^{n-2} |D u|^2 D u)&=f   &&\text{on}&&\Omega(t);  \\
u&= 0 &&\text{on}&&\partial\Omega(t);\\
u^n \frac{\partial\Delta u}{\partial\nu}&= 0 &&\text{on}&&\partial\Omega(t);\\
u(0)&=u_0  &&\text{on }&&\Omega &&
\end{aligned}\right.
\end{equation}
has at least one classical solution. 
The third condition reflects conservation of mass. This is a generalization of the problem studied in \cite{DalGiaGrun01, Grun02}.
The existence of a solution can be observed from the fact that the solution $\hat{u}$ to the first and fourth lines satisfies
$$ \hat{u}(t,\cdot) \in bc^{4+s,\vartheta}_\beta(\Omega)\cap U^{2+s}_\vartheta, \quad t\in \dot{J}. $$
Hence, for $t\in \dot{J}$ there are two continuous positive functions $c(t)<C(t)$ such that
\begin{equation}
\label{S4.4: lower bd of u}
c(t) \leq {\sf d}^{\beta\vartheta}(x) \hat{u}(t,x) \leq C(t) ,\quad  x\in\Omega, 
\end{equation}
and 
$${\sf d}^{(\vartheta-1)\beta}(x)\hat{u}^n(t,x) |D\Delta \hat{u} (t,x)|_{g_m} \leq C(t), \quad x\in\Omega.$$
The second inequality follows from \eqref{S4: power of u}, \eqref{S4: grad} and the fact that
$$\Delta \in \L(bc^{4+s,\vartheta}_\beta(\Omega), bc^{2+s,2 +\vartheta}_\beta(\Omega)). $$
The above two inequalities imply that for every $t\in\dot{J}$, as $x\to \partial\Omega$
$$|\hat{u}(t,x)|\leq C(t){\sf d}^{-\beta\vartheta}(x)\to 0, \quad  \hat{u}^n (t,x)|D\Delta \hat{u} (t,x)|_{g_m}  \leq C(t){\sf d}^{(1-\vartheta)\beta}(x)\to 0.$$
The fact that $\hat{u}(t,\cdot)>0$ on $\Omega$ is a consequence of \eqref{S4.4: lower bd of u}.
Therefore, 
\begin{equation}
\label{S4.4: supp u}
\supp[\hat{u}(t,\cdot)]=\Omega(t)=\Omega, \quad t\in J,
\end{equation}
and $\hat{u}$ is indeed a solution to equation~\eqref{S4.4: thin-film-IBVP}. If we seek solutions in the class
$$C^1_{1/2}(J(u_0),bc^{s,\vartheta}_\beta(\Omega))\cap C_{1/2}(J(u_0),bc^{4+s,\vartheta}_\beta(\Omega)) , $$
then $\hat{u}$ is actually the unique solution.
Note that the solution to equation~\eqref{S4.4: thin-film-IBVP} is, in general, not unique unless a third condition is prescribed on the free boundary $\partial[\supp(u)]$. A conventional supplementary condition is to set the contact angle to be zero.

By identifying $\hat{u},f,u_0\equiv 0$ on $\R^m\setminus \Omega$, $\hat{u}$ is nothing but a weak solution to the Cauchy problem
\begin{equation*}
\left\{\begin{aligned}
\partial_t u +\div (u^n D \Delta u +\alpha_1 u^{n-1}\Delta u D u +\alpha_2 u^{n-2} |D u|^2 D u)&=f    &&\text{on}&&\R^m_T;  \\
u(0)&=u_0   &&\text{on }&&\R^m&&
\end{aligned}\right.
\end{equation*}
belonging to the class $C_{1/2}(J; W^2_1(\R^m))$ for $\beta\in [1,n/(2n-4)]$ when $n\in (2,3]$, or for all $\beta\geq 1$ while $n\in (0,2]$ in the sense that
\begin{align*}
\int\limits_J\int\limits_{\R^m} \{ & u \partial_t\phi- \Delta u\div(u^n D \phi) + \alpha_1 u^{n-1} \Delta u (D\phi | D u)_{g_m}\\
& +\alpha_2 u^{n-2} |Du|^2 (D\phi |D u)_{g_m}   \}\, dx\, dt = -\int\limits_J\int\limits_{\R^m} f\phi\, dx\, dt 
\end{align*}
for all $\phi\in C_0(\mathring{J}; W^2_\infty(\R^m))$. To prove this statement, one first observes that, by the {\em uniform exterior and interior ball condition}, for some sufficiently small $a>0$ there is some $a$-tubular neighborhood of $\partial\Omega$, denoted by ${\sf T}_a$, such that ${\sf T}_a$ can be parameterized by 
$$
\Lambda: (-a,a)\times \partial\Omega \rightarrow {\sf T}_a: \quad (r,\p)\mapsto \p+ r \nu_\p,
$$
where $\nu_\p$ is the inward pointing unit normal of $\partial\Omega$ at $\p$. 
By the implicit function theorem, there exists some $C^5$-function $\Theta$ such that 
$$\Lambda^{-1}: {\sf T}_a \to  (-a,a)\times \partial\Omega,\quad \Lambda^{-1}(x)= ({\sf d}(x), \Theta(x)),$$
where ${\sf d}$ is defined in \eqref{S3.2: rescaled dist}, and $\Theta(x)$ is the closest point on $\partial\Omega$ to $x$.


To verify that $\hat{u}\in C_{1/2}(J; W^2_1(\R^m))$, it suffices to check the integrability of $\hat{u}$ near $\partial\Omega$. Since $u\in C_{1/2}(J, bc^{4+s,\vartheta}_\beta(\Omega))$, there exists a positive function $P\in C_{1/2}(J)$ such that
$$ {\sf d}^{(2 +\vartheta)\beta} (x)| \nabla^2 \hat{u}(t,x) |\leq P(t),\quad x\in \Omega, \quad t\in \dot{J}. $$
Then 
\begin{align*}
\int\limits_{{\sf T}_a} |\nabla^2 \hat{u}(t,x)|\, dx &\leq P(t) \int\limits_{{\sf T}_a \cap \Omega} {\sf d}^{-(2+\vartheta)\beta}(x)\, dx\\
&\leq M P(t) \int\limits_0^a \int\limits_{\partial\Omega}  r^{-(2+\vartheta)\beta} \, d\mu\, dr,
\end{align*}
which is finite iff $n\in (0,2]$, or $\beta\in [1,n/(2n-4)]$ and $n\in (2,3]$. The last line follows from the compactness of 
$\partial\Omega$ and \cite[formula~(25)]{PruSim13}. The argument for lower order derivatives of $\hat{u}$ is similar.

What is more, \eqref{S4.4: supp u} states that the support of $\hat{u}$ has the global small term waiting-time property for  all dimensions, that is, there exists some $T^*>0$ such that
\begin{equation}
\label{S4.4: waiting time}
\supp[u(t,\cdot)] =\supp[u(0,\cdot)], \quad t\in (0,T^*). 
\end{equation}
To the best of the author's knowledge, this is the first known result for the generalized thin film equation~\eqref{S4: thin-film-eq}. 
This result also supplements those in \cite{DalGiaGrun01, Grun02, Shish07} for the case dimension $m\geq 4$ with $n\in (0,3]$ and to domains without the {\em external cone property} with $n\in [2,3]$. For any $y\in \partial\Omega$, $\Omega$ is said to satisfy {\em  the external cone property} at $y$ if for some $\theta\in (0,\pi/4)$ there is an infinite cone $\mathcal{C}(y,\theta)$ with vertex $y$ and opening angle $\theta$ such that 
$$\supp[u_0]\cap \mathcal{C}(y,\theta)=\emptyset.$$ 
See \cite[Theorem~4.1]{Grun02} for more details. A domain $\Omega$ is said to enjoy {\em  the external cone property} if it satisfies this property at every $y\in \partial\Omega$. Note that any $u_0\in U^{2+s}_\vartheta$ fulfils the flatness condition of the initial data in \cite[Theorem~4.1]{Grun02}.
\end{remark}


\section*{Acknowledgements}
The authors would like to express his sincere gratitude to Prof. Herbert Amann for valuable suggestions on applications of the theory in this paper. I would also like to thank my thesis advisor, Gieri Simonett, for many helpful discussions.

\end{document}